\documentclass[a4paper,draft,reqno,12pt]{amsart}

\usepackage{cmap}
\usepackage[T2A]{fontenc}
\usepackage[utf8]{inputenc}
\usepackage[english]{babel}

\usepackage{amscd}
\usepackage{euscript}
\usepackage{amssymb}
\usepackage{amsmath}
\usepackage{amsthm}
\usepackage{enumitem}
\usepackage[hidelinks]{hyperref}
\usepackage[all]{xy}
\usepackage{secdot} %

\newcommand{\ZZ}{\mathbb{Z}}
\newcommand{\NN}{\mathbb{N}}

\newcommand{\KK}{\mathbb{K}}
\newcommand{\Aut}{\mathrm{Aut}}
\newcommand{\Ker}{\mathrm{Ker}}

\newcommand{\gr}{\mathrm{gr}}
\newcommand{\LT}{\mathrm{LT}}
\newcommand{\LM}{\mathrm{LM}}

\renewcommand{\phi}{\varphi}

\theoremstyle{plain}
\newtheorem{theorem}{Theorem}[section]
\newtheorem{proposition}[theorem]{Proposition}
\newtheorem{lemma}[theorem]{Lemma}

\theoremstyle{definition}
\newtheorem{definition}[theorem]{Definition}

\theoremstyle{remark}
\newtheorem{remark}[theorem]{Remark}

\usepackage[a4paper]{geometry}
\begin{document}

\date{}

\title{Isomorphism classes of multi Danielewski varieties}

\author{P.~Evdokimova}
\address{Lomonosov Moscow State University, Faculty of Mechanics and Mathematics, Department of Higher Algebra, Leninskie Gory 1, Moscow, 119991 Russia;
\linebreak and \linebreak
HSE University, Faculty of Computer Science, Pokrovsky Boulvard 11, Moscow, 109028 Russia.}
\email{polina.evdokimova@math.msu.ru}

\thanks{The work was supported by the Theoretical Physics and Mathematics Advancement Foundation <<BASIS>>, project number 24-7-2-20-1.}


\sloppy

\maketitle

\begin{abstract}
We study multi Danielewski varieties, which are the generalization of Danielewski varieties and Double Danielewski varieties. The paper contains description of the Makar-Limanov invariant of these varieties and of their isomorphism classes.
\end{abstract}

\section{Introduction}

Let $\KK$ be an algebraically closed field of characteristic zero, and let $X$ and $Y$ be irreducible affine algebraic varieties over $\KK$. The {\it Generalized Cancellation Problem}, formulated by O. Zariski, asks whether $X \times \KK \simeq Y \times \KK$ implies $X \simeq Y$. Several counterexamples to this conjecture are known. The first counterexample was constructed by Danielewski. In 1989, in his paper~\cite{D}, he introduced a class of varieties defined by equations of the form $xy^n = f(z)$. Such varieties are now called {\it Danielewski surfaces}. For $n = 1$ and $n = 2$, Danielewski surfaces provide a counterexample to the Generalized Cancellation Problem. See, e.g.,~\cite{M1} and~\cite{M2}, where it was proved that the varieties $$\mathbb{V}(xy - f(z)) \; \; \text{and} \; \; \mathbb{V}(x^2y - f(z))$$ are not isomorphic by means of the {\it Makar-Limanov invariant}. The Makar-Limanov invariant was first introduced in~\cite{M3}, where it was called the <<ring of absolute constants>> or the <<absolute kernel>>. This invariant often allows one to prove that varieties are non-isomorphic. For instance, it was used to show that the Koras–Russell cubic is not isomorphic to an affine space; see~\cite{M3}.

In~\cite{M1} and~\cite{M2}, Makar-Limanov in particular computes generators of the automorphism groups of Danielewski surfaces. These works motivated the study of automorphisms of Danielewski surfaces and their generalizations. For example, in~\cite{C} the automorphism group of the variety $$\mathbb{V}(xy^n - z^2 - h(y)z)$$ over an arbitrary field is described. In~\cite{DP}, a more general class of surfaces defined by the equation $$x^nz - Q(x,y)= 0, \; \; \text{where} \; \; n \geq 2,$$ and $Q(x,y)$ is a polynomial with coefficients in an arbitrary field, was considered. For such surfaces, generators of their automorphism groups were computed. Later, a complete classification of varieties of this type up to isomorphism and automorphism of an ambient affine space was obtained; see~\cite{P}. Another generalization of Danielewski surfaces was studied in~\cite{BV}: for surfaces of the form $$\mathbb{V}(xf(x) - Q(y,z)),$$ $$\text{where} \; \; \deg f \geq 2, d \geq 2, \; \; \; \; Q(y,z) = z^d + s_{d-1}(y)z^{d-1} + \ldots + s_0(y),$$ the Makar-Limanov invariant and generators of the automorphism group were computed.

In 2007, another generalization of Danielewski surfaces was proposed in~\cite{Du}. Namely, varieties defined by equations of the form
$$xy_1^{k_1}\ldots y_m^{k_m} = z^d + s_{d-1}(y_1, \ldots, y_m)z^{d-1} + \ldots + s_0(y_1, \ldots, y_m).$$
were considered. Such varieties are called {\it Danielewski varieties}. By computing the Makar-Limanov invariant for these varieties, counterexamples to the Generalized Cancellation Problem in an arbitrary dimension were obtained. In~\cite{G}, the automorphisms of Danielewski varieties were described.

In~\cite{GS}, varieties defined by the pair of equations
$$\{x^dy - P(x, z) = 0, x^et - Q(x, y, z) = 0\}$$
over a field of an arbitrary characteristic were considered. Such varieties were called {\it double Danielewski surfaces}. The Makar-Limanov invariant was computed for them and their isomorphism classes were described. Moreover, it was proved in~\cite{GS} that double Danielewski surfaces provide counterexamples to the Generalized Cancellation Problem; see~\cite[Corollary 3.15]{GS}. 

In \cite{GP} the class of {\it multi Danielewski varieties} is introduced. This class contains  both the classes of double Danielewski surfaces and the classes of Danielewski varieties. A milti Danielewski variety is given by a systems of equations of the following form:
\begin{equation*}
\begin{cases}
    x_1y_1^{k_{11}} \ldots y_m^{k_{1m}} = f_1(z, y_1, \ldots, y_m),\\
    x_2y_1^{k_{21}} \ldots y_m^{k_{2m}} = f_2(z, y_1, \ldots, y_m, x_1),\\
     \; \;\; \;\; \;\; \;\; \;\; \;\; \;\; \;\; \;\; \;\; \;\; \;\ldots\\
    x_ny_1^{k_{n1}} \ldots y_m^{k_{nm}} = f_n(z, y_1, \ldots, y_m, x_1, \ldots, x_{n-1}).
\end{cases}
\end{equation*}

It is proved in~\cite[Example~4.14]{GP} that under some conditions $X\times \KK\cong Y\times \KK$, where $X$ and $Y$ are multi Danielewski varieties differing from each other only by exponents $k_{nj}$ in the last equation. Therefore, to prove that multi Danielewski varieties give new counterexamples to the Generalized Cancellation Problem one should prove that $X$ and $Y$ are not isomorphic. 

In this paper we classify multi Danielewski varieties (satisfying some very slight conditions on the parameters) up to isomorphism. Firstly, we compute the Makar-Limanov invariant of multi Danielewski varieties for $n =2$ in Sections~\ref{Prl} and~\ref{sML}; see Theorem~\ref{ML}. Theorem~\ref{main1} gives a classification of multi Danielewski varieties for $n=2$ up to isomorphism. Section~\ref{Gc} covers general case, when $n$ is an arbitrary positive integer: using the previous results in general case we also describe the Makar-Limanov invariant and obtain a classification of multi Danielewski varieties, see Theorems~\ref{ML1} and~\ref{main111}.

The author is grateful to S. Gaifullin for statement of the problem and continuous assistance during the work on the paper.

\section{Locally nilpotent derivations}

Some preliminaries on the field of locally nilpotent derivations are gathered in this section. Detailed information on this area can be found, for example, in~\cite{F}.

Let $A$ be a finitely generated algebra over a field $\KK$. By the word <<algebra>> from now on we mean unital commutative associative algebra over the field $\KK$.

\begin{definition}
    A {\it derivation} of an algebra $A$ is a linear operator $\delta\colon A\rightarrow A$, satisfying the Leibniz rule:
    $\delta(fg)=f\delta(g)+g\delta(f)$.
\end{definition}

\begin{definition}
A derivation $\delta\colon A\rightarrow A$ is called {\it locally nilpotent} (LND) if for every $f\in A$ there is a non-negative integer $n$ such that $\delta^n(f)=0$.
\end{definition}
\noindent We will denote the set of locally nilpotent derivations of $A$ by $\mathrm{LND}(A)$.

Denote the transcendence degree of an algebra $A$ over $\KK$ by $\mathrm{tr.deg.}A$.

\begin{proposition}\label{LND} \cite[Principle 11 (e)]{F}
Let $A$ be an integral domain. Fix an arbitrary LND $\delta$ and introduce the notation $B := \Ker(\delta)$. Then $B$ is a factorially closed subalgebra of $A$. Moreover, the following is true: $\mathrm{tr.deg.}B = \mathrm{tr.deg.}A - 1$.
\end{proposition}

Consider a grading on an algebra $A$ by an abelian group $G$
$$
A=\bigoplus_{g\in G}A_g.
$$
A derivation $\delta$ is called homogeneous if it sends homogeneous elements to homogeneous ones. It follows from the Leibniz rule that for any nonzero homogeneous derivation $\delta$ there is an element $g_0\in G$ such that $\delta$ sends $A_g$ to $A_{g+g_0}$. An element $g_0$ is called the {\it degree} of the derivation $\delta$. It is easy to prove that any derivation can be decomposed into a finite sum of homogeneous ones. We call these homogeneous summands {\it homogeneous components} of the derivation. 

The following lemma is proved in \cite{R}:
\begin{lemma}\label{hl}
    Consider a $\ZZ$-grading on $A$. If $\delta \in \mathrm{LND}(A)$ and $\delta = \sum_{i = k}^{l} \delta_i$, then $\delta_k$ and $\delta_l$ are locally nilpotent.  
\end{lemma}
Now, suppose there is a $\ZZ^n$-grading and $\delta$ is decomposed into a sum of homogeneous derivations with respect to the
grading. Then the convex hull of the degrees of homogeneous components forms a polyhedron. In this case it turns out that the components that correspond to the vertices of the polyhedron are locally nilpotent derivations. So if there is a nonzero LND $\delta$ of $A$, then there is a nonzero homogeneous LND of $A$, namely, the component corresponding to one of the vertices of the polyhedron.
Note that if $\delta$ is a homogeneous LND of a graded algebra $A$, then $\Ker(\delta)$ is a graded subalgebra of $A$.

Given an LND $\delta$, one can construct its {\it exponent} that is defined by the series:
 $$\exp(\delta) := \sum_{i \geqslant 0} \frac{\delta^i}{i!}.$$
 Applying $\exp(\delta)$ to any element $a\in A$, one obtains a finite sum since $\delta$ is locally nilpotent. Hence, an exponent $\exp(\delta)$ is correctly defined. It is easy to show that $\exp(\delta)$ is an automorphism of the algebra $A$. Consider algebraic subgroups of the group $\mathrm{Aut}(A)$ of automorphisms of the algebra $A$, isomorphic to the additive group of the ground field $\KK$. We shall call such subgroups $\mathbb{G}_a$-subgroups. Each LND $\delta$ corresponds to the following subgroup $$\mathcal{H}_\delta=\{ \exp(t\delta), t \in\KK \}.$$ This correspondence sets a bijection between LNDs that are considered up to the proportionality with a scalar factor and $\mathbb{G}_a$-subgroups in $\Aut(A)$.
 
\begin{definition}
    A {\it degree} of an element $f \in A, f \neq 0,$ with respect to LND $\delta$ is the number $\deg_{\delta}(f) := \min \{n \in \NN \; | \; \delta^{n+1}(f) = 0\}$ and $\deg_{\delta}(0) := -\infty$.
\end{definition}

\begin{definition}
    The {\it Makar-Limanov Invariant} of an algebra $A$ is a subalgebra of $A$, defined by: $$\mathrm{ML}(A) = \bigcap_{\delta \in \mathrm{LND}(A)} \Ker(\delta).$$
\end{definition}

\begin{definition}
    An algebra $A$ is called {\it rigid} if $\mathrm{ML}(A) = A$. In other words, there are no nonzero LNDs of $A$.
\end{definition}
\noindent Equivalently, there are no $\mathbb{G}_a$-subgroups in the group $\Aut(A)$.
\begin{definition}
    LNDs $\delta$ and $\partial$ are called {\it equivalent} if $\Ker(\delta) = \Ker(\partial)$.
\end{definition}

\begin{definition}
    An algebra $A$ is called {\it semi-rigid} if all LNDs of $A$ are equivalent.
\end{definition}
\noindent Note that if for some $\delta \in$ $\mathrm{LND}(A)$ holds $\Ker(\delta) = \mathrm{ML}(\delta)$, then $A$ is semi-rigid.
\begin{definition}
    An affine algebraic variety $X$ is {\it rigid} if its algebra of regular functions is rigid.
\end{definition}

\section{Filtrations}
Let $A$ be an algebra over a field $\KK$. 
\begin{definition}
    A set $\{A_n\}_{n \in \ZZ}$ of vector spaces $A$ over a field $\KK$ is called a {\it $\ZZ$-filtration}, if The following holds:
    \begin{enumerate}
        \item[(1)] $A_n \subseteq A_{n+1}$ for all $n \in \ZZ$,
        \item[(2)] $A = \bigcup_{n \in \ZZ}A_n$,
        \item[(3)] $\bigcap_{n \in \ZZ} A_n = \{0\}$,
        \item[(4)] $(A_n \setminus A_{n-1}) \cdot (A_m \setminus A_{m-1}) \subseteq A_{n+m} \setminus A_{n+m-1}$ for all $n, m \in \ZZ$.
    \end{enumerate}
\end{definition}
\begin{definition}
    $\ZZ$-filtration is called {\it admissible}, if there exists a finite generating set $\Gamma$ of $A$ such that, for any $n \in \ZZ$ and $a \in A_n$, $a$ can be written as a finite sum of monomials in elements of $\Gamma$ and each of these monomials is an element of $A_n$.
\end{definition}
Let $A$ be an algebra of regular functions of some affine algebraic variety $X$. Then
$$A = \KK[X] \simeq \KK[X_1, \ldots, X_n]/I,$$
where $I$ is an ideal, consisting of polynomials, that take a zero value on $X$.
Fix a tuple of integers $\omega = (\omega_1, \ldots, \omega_n)$. We use this tuple to construct a filtration of $A$ as follows. A nonzero polynomial $F \in \KK[X_1, \ldots, X_n]$ is the sum of some monomials with coefficients in $\KK^{\times}. $ Denote the set of these monomials by $M(F)$. Introduce a degree function $\deg_{\omega}$ on $\KK[X_1, \ldots, X_n]$. For every $F \in \KK[X_1, \ldots, X_n]$ put
$$\deg_{\omega}(F) := \max\{\deg_{\omega}(G) \; | \; G \in M(F)\}, \; \text{if} \; F \neq 0,$$
$$\deg_{\omega}(0) := - \infty,$$
where $\deg_{\omega}(X_1^{a_1} \ldots X_n^{a_n}) := a_1 \omega_1 + \ldots + a_n \omega_n$ for any $a_i \in \ZZ_{\geq0}$.\\
Let $\pi$ denote a homomorphism $\pi: \KK[X_1, \ldots, X_n] \rightarrow A$. Then for every $f \in A$ put
$$d_{\omega}(f) := \min_{F \in \pi^{-1}(f)}\deg_{\omega}F, \; \text{if} \; f \neq 0,$$
$$d_{\omega}(0) := - \infty.$$
Note that the function $d_{\omega} : A \rightarrow \ZZ \cup \{-\infty\}$ is {\it semi-degree function}, that is for all $a, b \in A$ the following holds:
\begin{enumerate}
    \item[(1)] $d_{\omega}(a+b) \leq \max\{d_{\omega}(a), d_{\omega}(b)\},$
    \item[(2)] $d_{\omega}(ab) \leq d_{\omega}(a) + d_{\omega}(b),$
    \item[(3)] $d_{\omega}(a) = -\infty \Leftrightarrow a = 0.$ 
\end{enumerate}
Moreover, a collection of subsets $A_n = \{a \in A \; | \; d_{\omega}(a) \leq n\}$ is a filtration. We will call this filtration of $A$, constructed via $\omega$, {\it $\omega$-weight filtration}.
\section{Associated graded algebras}
Let $\{A_n\}_{n \in \ZZ}$ be a $\ZZ$-filtration of an algebra $A$.
\begin{definition}
    {\it Associated graded algebra} of an algebra $A$ is a vector space $$\gr(A) := \bigoplus_{i\in \ZZ}A_i /A_{i-1}$$
    with multiplication, which is defined on homogeneous elements
   $$a + A_{i-1} \in A_i/A_{i-1} \; \; \text{and} \; \; b + A_{j-1} \in A_j/A_{j-1}$$
    by the following rule:
    $$(a + A_{i-1}) \cdot(b + A_{j-1}) = ab + A_{i+j-1} \in A_{i+j}/A_{i+j-1}.$$
    It extends to the rest of elements via distributivity.
\end{definition}
Define a map $\rho : A \rightarrow \gr(A)$ as follows:
$$\rho(a) = a + A_{n-1}, \; \text{ if } \; a \in A_n \setminus A_{n-1}.$$
We will call this map {\it a natural map from an algebra to the associated graded algebra}. Note that $\rho$ may not be an algebra homomorphism.
Let $\delta$ be an LND of an algebra $A$.
\begin{definition}
    An LND $\delta$ {\it preserves a filtration $\{A_n\}_{n \in \ZZ}$}, if there is $i_0 \in \ZZ$ such that for every $i \in \ZZ$ we have $\delta(A_i) \subseteq A_{i + i_0}$.
\end{definition}

\noindent Define a derivation $\gr(\delta)$ of an algebra $\gr(A)$ for those derivations $\delta$ of an algebra $A$ that preserve a filtration. We set:
$$\gr(\delta) := 0, \; \text{ if } \; \delta \equiv 0,$$
$$\gr(\delta): A_i/A_{i-1} \rightarrow A_{i+t}/A_{i+t-1}, \; a + A_{i-1} \longmapsto \delta(a) + A_{i+t-1},$$
where $a$ is a homogeneous element and $t$ is minimal integer such that $\delta(A_i) \subseteq A_{i+t}$. The definition extends to the rest of elements of $\gr(A)$ via linearity. Note that a derivation $\gr(\delta)$ is locally nilpotent, see~\cite[Principle 15]{F}. Moreover, if $A$ is a finitely generated algebra then any LND of $A$ preserves any filtration, see~\cite[Lemma 6.2]{CS}.

\begin{proposition}\cite[Corollary 6.3]{CS}\label{grrig}
    Let $A = \KK[X_1, \ldots, X_n] / I$ be an algebra with $\omega$-weight filtration $\{A_i\}_{i \in \ZZ}$ and $\gr(A)$ be an associated graded algebra, constructed on this filtration. Then a rigidity of $\gr(A)$ implies a rigidity of $A$.
\end{proposition}

\begin{theorem}\label{homgr} \cite[Proposition 2.2]{DHML}
    Let $A$ be an integral domain over a field $\KK$ with an admissible $\ZZ$-filtration and $\gr(A)$ is the induced $\ZZ$-graded domain, $\delta$ is nonzero LND of $A$. Then $\delta$ induces a nonzero LND $\overline \delta$ of $\gr(A)$ such that $$\rho(\Ker(\delta)) \subseteq \Ker(\overline \delta).$$
\end{theorem}

\section{Groebner bases}
Let us briefly recall definition and basic facts on the field of Groebner bases. For more details, see~\cite{CLSh}.

\begin{definition}
    A {\it monomial order} on $\KK[X_1, \ldots, X_n]$ is a total order $<$ on $\ZZ^n_{\geq 0}$ such that for any $\alpha, \beta, \gamma \in \ZZ^n_{\geq 0}$ the relation $\alpha < \beta$ implies $a + \gamma < \beta + \gamma$.
\end{definition}

\begin{definition}
    A {\it lexicographic order} is a monomial ordering such that $\alpha < \beta$ if and only if the left-most nonzero coordinate of a difference $\alpha - \beta$ is positive. We use the notation: $\alpha <_{lex} \beta$.
\end{definition}

\begin{definition}
    {\it $\omega$-graded lexicographic order} is a monomial order such that $\alpha < \beta$ if and only if $\sum\omega_i \alpha_i < \sum\omega_i \beta_i$ or $\sum\omega_i \alpha_i < \sum\omega_i \beta_i$ and $\alpha <_{lex} \beta$. We use the notation: $\alpha <_{\omega-grlex} \beta$.
\end{definition}
Set a tuple $\alpha := (\alpha_1, \ldots, \alpha_n) \in \ZZ^n_{\geq 0}$. Define $X^{\alpha} := X^{\alpha_1} \ldots X^{\alpha_n}$ and let $F = \sum a_{\alpha}X^{\alpha}$ be a nonzero polynomial. Set a monomial ordering $<$ and put
$$\mathrm{multideg}(F) := \max\{\alpha \; | \; a_{\alpha \neq 0}\},$$
where the maximum is taken with respect to the order $<$. The {\it leading term} of a polynomial $F$ is $\mathrm{LT}(F) := a_{\mathrm{multideg}(F)}X^{\mathrm{multideg}(F)}$. The {\it leading monomial} of a polynomial $F$ is $\mathrm{LM}(F) := X^{\mathrm{multideg}(F)}$.

\begin{definition}
    The $S$-polynomial of $F_i$, $F_j \in \mathcal{F}$ is a polynomial of the following form:
    $$S(F_i, F_j) := \frac{m}{\LT(F_i)}F_i - \frac{m}{\LT(F_j)}F_j,$$
    where $m :=$ LCM$(\LM(F_i), \LM(F_j))$.
\end{definition}

\begin{definition}
    Let $\mathcal{F} = \{F_1, \ldots, F_m\}$ be an ordered with respect to monomial ordering $<$ tuple of polynomials in $\KK[X_1, \ldots, X_n]$. Then each polynomial $F \in \KK[X_1, \ldots, X_n]$ can be written as
    $$F = a_1F_1 + \ldots a_mF_m + r,$$
    for $a_i, r \in \KK[X_1, \ldots, X_n]$, and either $r =0$, or $r$ is a linear combination of monomials with coefficients in $\KK^{\times}$, and neither of the coefficients is divisible by $\LT(F_i)$ for all $i$. Then $r$ is called a {\it remainder of a polynomial $F$ on division by $\mathcal{F}$}. See.~\cite[Chapter 2, §3, Theorem 3]{CLSh}.
\end{definition}

\begin{definition}
    A tuple of polynomials $\mathcal{F} = \{F_1, \ldots, F_m\}$, where $F_i \in \KK[X_1, \ldots, X_n]$ is a {\it Groebner system}, if for any $G \in \KK[X_1, \ldots, X_n]$ a remainder of $G$ on division by $\mathcal{F}$ is uniquely defined.
\end{definition}

\begin{definition}
    A tuple of polynomials $\mathcal{F} = \{F_1, \ldots, F_m\}$ is a {\it Groebner basis} of an ideal $I \triangleleft \KK[X_1, \ldots, X_n]$, if $(\mathcal{F}) = I$ and $\mathcal{F}$ is a Groebner system.
\end{definition}

\begin{lemma}\label{s}
    Let $F_1, F_2 \in \KK[X_1,\ldots,X_n]$, and $\overline{F_1}$, $\overline{F_2}$ are relatively prime. Then there is a sequence of elementary divisions of $S$-polynomial $S(F_1, F_2)$ by $\{F_1, F_2\}$ such that the remainder of the sequence is zero.
\end{lemma}

\begin{theorem}[Buchberger’s Criterion]\label{Buh}
    Let $\mathcal{F} = \{F_1, \ldots, F_m\}$, where $F_i \in \KK[X_1,\ldots,X_n]$ for all $i$. Then $\mathcal{F}$ is a Groebner system if and only if for any two polynomials $F_i$ and $F_j$ the remainder on elementary division of $S(F_i, F_j)$ by $\mathcal{F}$ is zero.
\end{theorem}

Fix a tuple $\omega = (\omega_1, \ldots, \omega_n) \in \ZZ^n$ and let $\{A_i\}_{i \in \ZZ}$ be an $\omega$-wight filtration of an algebra $A = \KK[X_1, \ldots, X_n]/I$. Any polynomial $F \in \KK[X_1, \ldots, X_n]$ can be written as a sum of $\deg_{\omega}$-homogeneous elements $F = F_{i_1} + \ldots + F_{i_k}$, where $\deg_{\omega}(F_{i_j}) = i_j$ and $i_1 < \ldots < i_k = \deg_{\omega}(F)$. Denote the leading term in this sum by $\overline F^{\omega} := F_{i_k}$. Using~\cite[Proposition 4.1]{KM}, it is easy to prove the following fact:

\begin{proposition}\label{Grebner}
    Let $A = \KK[X_1, \ldots, X_n]/I$ and $F_1, \ldots, F_m$ be a Groebner basis of an ideal $I$ with respect to a lexicographic order, constructed by $\omega = (\omega_1, \ldots \omega_n)$. Then $$\gr(A) = \KK[X_1, \ldots, X_n]/(\overline{F_1}^\omega, \ldots, \overline{F_m}^\omega).$$
\end{proposition}

\section{Preliminary lemmas}\label{Prl}

Introduce some notations. Let $A, \hat{A}$ and $B$ denote the following algebras: 
$$A := \KK[Z, Y_1, Y_2, \ldots, Y_m, X_1, X_2],$$
$$\widehat{A} := \KK[Z, Y_1, Y_{1}^{-1}, Y_2, Y_2^{-1}, \ldots, Y_m, Y_m^{-1}],$$
$$B := \frac{A}{(X_1Y_1^{k_{11}}\ldots Y_m^{k_{1m}} - f_1(Z, Y_1, \ldots, Y_m), X_2Y_1^{k_{21}}\ldots Y_m^{k_{2m}} - f_2(Z, Y_1, \ldots, Y_m, X_1))},$$
 where $k_{ij} \in \NN$, a polynomial $f_1(Z, Y_1, \ldots, Y_m) \in \KK[Z, Y_1, \ldots, Y_m]$ is monic in a variable $Z$; polynomial $f_2(Z, Y_1, \ldots, Y_m, X_1) \in \KK[Z, Y_1, \ldots, Y_m, X_1]$ is monic in a variable $X_1$. The letters $z, y_1, \ldots, y_m, x_1, x_2$ will denote the images of $Z, Y_1, \ldots, Y_m, X_1, X_2$ in $B$ respectively.

Set $\deg_Zf_1(Z, Y_1, \ldots, Y_m) := r, \deg_{X_1}f_2(Z, Y_1, \ldots, Y_m, X_1) := s$. We will prove a generalization of~\cite[Theorem 3.10]{GS}.

\begin{lemma}\label{id}
    Let $R$ be an integral domain, $y_1, y_2, \ldots, y_m, f \in R \setminus \{0\}$, and $y_1, y_2, \ldots, y_m$ are not zero-divisors on $R/(f)$. Then $$\frac{R[U]}{y_1y_2 \ldots y_mU - f} \; \; \; \text{is an integral domain}.$$
\end{lemma}

\begin{proof}
    It is easy to see that $y_1, \ldots, y_m$ are not zero-divisors on $R[U]/(y_1 \ldots y_mU - f)$, since $y_1, \ldots, y_m$ are not zero-divisors on $R/(f)$. Indeed, assume that $y_i \beta = (y_1 \ldots y_mU - f) \alpha$ for some $\alpha, \beta \in R$. Thus $y_i(y_1 \ldots y_{i-1}y_{i+1}\ldots y_mU - \beta) = f \alpha$. Since $y_i$ is not a zero-divisor on $R/(f)$ and $R$ is an integral domain, then $f$ is not a zero-divisor on $R/(y_i)$ and on $R[U]/(y_i)$, see~\cite[Lemma~3.1]{GS}. Therefore, $\alpha \in y_iR[U]$, hence $y_1 \ldots y_mU - f$ is not a zero-divisor on $R[U]/(y_i)$. It follows that $y_i$ is not a zero-divisor on $R[U]/(y_1 \ldots y_mU - f)$.
    
    We obtain an embedding:
    $$R[U]/(y_1 \ldots y_mU - f) \hookrightarrow \frac{R[U][y_1^{-1}]\ldots[y_m^{-1}]}{(y_1 \ldots y_mU - f)} = R[y_1^{-1}]\ldots[y_m^{-1}].$$
    Hence, $R[U]/(y_1 \ldots y_mU - f)$ is an integral domain.
\end{proof}

\begin{lemma}\label{id2}
    $B$ is an integral domain.
\end{lemma}
\begin{proof}
    Set $$R := \frac{\KK[Z, Y_1, Y_2, \ldots, Y_m, X_1]}{(X_1Y_1^{k_{11}}\ldots Y_m^{k_{1m}} - f_1(Z, Y_1, \ldots, Y_m))}.$$ Then in the same manner as in ~\cite[Lemma 3.3]{GS} we get the fact that $R$ is an integral domain. We can identify $R$ as a subring of $B$ by identifying the images of $Z, Y_1, \ldots, Y_m, X_1$ in $R$ with $z, y_1, \ldots, y_m, x_1$ in $B$. Then $$B = R[X_2]/(X_2y_1^{k_{21}}\ldots y_m^{k_{2m}} - f_2(z, y_1, \ldots, y_m, x_1)).$$
Moreover, $R/(y_1) \simeq \left(\frac{\KK[Z, Y_2, \ldots, Y_m]}{f_1(Z, 0, Y_2, \ldots, Y_m)}\right)[X_1]$. Therefore, since a polynomial $f_2$ is monic in $X_1$, it follows that its image in $R/(y_1)$ is not a zero-divisor. Hence, $y_1^{k_{21}}$ is not a zero-divisor on $R/(f_2(z, y_1, \ldots, y_m, x_1))$ by~\cite[Lemma~3.1]{GS}. Proceeding in the same manner, we obtain the fact that $y_2^{k_{22}}, \ldots, y_m^{k_{2m}}$ are not zero-divisors on $R/(f_2(z, y_1, \ldots, y_m, x_1))$. Hence, $B$ is an integral domain by Lemma~\ref{id}.
\end{proof}

\section{The Makar-Limanov invariant of the algebra B}\label{sML}

\begin{lemma}\label{l1}
    Put $D(1) := B$ and
     $$D(j) := \frac{A}{(X_1Y_1^{k_{11}}\ldots Y_m^{k_{1m}} - f_1(Z, 0, \ldots, 0, Y_{j}, \ldots, Y_m), X_2Y_1^{k_{21}}\ldots Y_m^{k_{2m}} - X_1^s)}, \;\; \text{if} \;\; j > 1.$$
     Consider $D(j)$, for every $j$, as a subalgebra of $\ZZ^m$-graded algebra $\hat{A} = \bigoplus_{i\in \ZZ^m}\KK[z]y^i$, where $y^i = y_1^{i_1}\cdots y_m^{i_m}$ and $i_k \in \ZZ$.

    Define a $\ZZ-$filtration $\{D(j)_n\}$ on $D(j)$ as follows:
    $$D(j)_n := D(j) \cap (\oplus_{i \geq -n}\KK[z, y_1, y_1^{-1}, \ldots, y_{j-1}, y^{-1}_{j-1}, y_{j+1}, y^{-1}_{j+1}, \ldots y_m, y_m^{-1}]y_j^i).$$
     
    This filtration is admissible and the corresponding graded ring $\gr(D(j)) \simeq D(j+1)$.
\end{lemma}

\begin{remark}
     In this lemma and in the following proof we assume that $z, y_1, \ldots, y_m, x_1, x_2$ denote for every $D(j)$  the images of $Z, Y_1, \ldots, Y_m, X_1, X_2$ in $D(j)$ respectively.
\end{remark}

\begin{proof}[Proof of Lemma~\ref{l1}]
    Note that $y_j \in D(j)_{-1} \setminus D(j)_{-2}$, $z, y_1, \ldots, y_{j-1}, y_{j+1}, \ldots, y_m \in  D(j)_{0} \setminus D(j)_{-1}$ and $ x_1 \in D(j)_{k_{1j}} \setminus D(j)_{k_{1j}-1}, x_2 \in D(j)_{k_{1j}s+k_{2j}} \setminus D(j)_{k_{1j}s+k_{2j}-1}$.
    Using the relations 
    $$x_1y_1^{k_{11}}\ldots y_m^{k_{1m}} = f_1(z, 0, \ldots, 0, y_{j}, \ldots, y_m)$$ $$\text{and} \;\;\; x_2y_1^{k_{21}}\ldots y_m^{k_{2m}} = f_2(z, y_1, \ldots, y_m, x_1), \; \text{if} \; D(j) = B,$$ $$x_2y_1^{k_{21}}y_2^{k_{22}}\ldots y_m^{k_{2m}} = x_1^s, \; \text{if} \; D(j) \neq B,$$ one can see that each element $g \in D(j)$ can be written as
    \begin{multline}\label{el1}
        g = f_0(y_1,\ldots y_m, z) + \sum_{\substack{i<k_{1j}, \\t > 0, \\ l_a}}b_{itl_a}(z)y_j^ix_1^ty_1^{l_1}\ldots y_{j-1}^{l_{j-1}}y_{j+1}^{l_{j+1}}\ldots y_m^{l_m} + \\ + \sum_{\substack{i \geq k_{1j}, \\t > 0, \\ l_a}}\hat{b}_{itl_a}(z)y_j^ix_1^ty_1^{l_1}\ldots y_{j-1}^{l_{j-1}}y_{j+1}^{l_{j+1}}\ldots y_m^{l_m} + \sum_{\substack{i<k_{2j}, \\t > 0, \\ l_a}}c_{itl_a}(z)y_j^ix_2^ty_1^{l_1}\ldots y_{j-1}^{l_{j-1}}y_{j+1}^{l_{j+1}}\ldots y_m^{l_m} + \\ +\sum_{\substack{i \geq k_{2j}, \\t > 0, \\ l_a}} \widehat{c}_{itl_a}(z)y_j^ix_2^ty_1^{l_1}\ldots y_{j-1}^{l_{j-1}}y_{j+1}^{l_{j+1}}\ldots y_m^{l_m} + \\+ \sum_{\substack{i < \max(k_{1j},k_{2j}), \\p > 0, t > 0,\\ l_a}} d_{itl_a}(z)y_j^ix_1^px_2^ty_1^{l_1}\ldots y_{j-1}^{l_{j-1}}y_{j+1}^{l_{j+1}}\ldots y_m^{l_m} + \\ + \sum_{\substack{i \geq \max(k_{1j},k_{2j}), \\p > 0, t > 0,\\ l_a}} \widehat{d}_{itl_a}(z)y_j^ix_1^px_2^ty_1^{l_1}\ldots y_{j-1}^{l_{j-1}}y_{j+1}^{l_{j+1}}\ldots y_m^{l_m},
    \end{multline}
    where in the sum with the coefficients $\widehat{b}_{itl_a}$ we sum up over the tuples $l_1,\ldots l_m$ such that there is an index $q$ with the property $l_q < k_{1q}$; in the sum with the coefficients $\hat{c}_{itl_a}$ we sum up over the tuples $l_1,\ldots l_m$ such that there is and index $q$ with the property $l_q < k_{2q}$; and in the sum with the coefficients $\widehat{d}_{itl_a}$ we sum up over the tuples $l_1,\ldots l_m$ such that there is and index $q$ with the property $l_q < \max(k_{1q},k_{2q})$. Moreover, $f_0(y_1,\ldots y_m, z) \in \KK[z, y_1, y_2, \ldots, y_m],$ and $ b_{itl_a}(z), \widehat{b}_{itl_a}(z), c_{itl_a}(z), \widehat{c}_{itl_a}(z), d_{itl_a}(z), \widehat{d}_{itl_a}(z) \in \KK[z]$.
    It follows from~(\ref{el1}), that the filtration defined on $D({j})$ is admissible with the generating set $\Gamma = \{z, y_1, \ldots, y_m, x_1, x_2\}$. Set $\widetilde D(j) := gr(D(j))$. Hence, $\widetilde D(j)$ is generated by $\widetilde z, \widetilde y_1, \ldots, \widetilde y_m, \widetilde x_1, \widetilde x_2$, i.e. by the images of the corresponding elements in $D(j)$.

    We show that $\widetilde D(j) \simeq D(j+1)$. It is easy to see that polynomials $x_1y_1^{k_{11}}y_2^{k_{12}}\ldots y_m^{k_{1m}}$ and $ f_1(z, 0, \ldots, 0, y_{j+1}, \ldots, y_m)$ are the elements of $ \in D(j)_0$. Moreover, the following relation holds: 
    \begin{multline*}
        x_1y_1^{k_{11}}y_2^{k_{12}}\ldots y_m^{k_{1m}} - f_1(z, 0, \ldots, 0, y_{j+1}, \ldots, y_m) = \\ = f_1(z, 0, \ldots, 0, y_j, \ldots, y_m) - f_1(z, 0, \ldots, 0, y_{j}, \ldots, y_m) \in D^{j}_{-1},
    \end{multline*}
    hence, $x_1y_1^{k_{11}}y_2^{k_{12}}\ldots y_m^{k_{1m}} - f_1(z, 0, \ldots, 0, y_{j+1}, \ldots, y_m) = 0 \in \widetilde D(j)$.

    Analogously, $x_2y_1^{k_{21}}y_2^{k_{22}}\ldots y_m^{k_{2m}}, x_1^s \in B_{k_{1j}s}$. If $D(j) = B$, then $x_2y_1^{k_{21}}y_2^{k_{22}}\ldots y_m^{k_{2m}} - x_1^s = f_2(z, y_1, \ldots, y_m, x_1) - x_1^s \in B_{k_{1j}s - 1}$, and $x_2y_1^{k_{21}}y_2^{k_{22}}\ldots y_m^{k_{2m}} - x_1^s = 0$ if $D(j) \neq B$. Thus $\widetilde x_2 \widetilde y_1^{k_{21}} \widetilde y_2^{k_{22}}\ldots \widetilde y_m^{k_{2m}} - \widetilde x_1^s = 0 \in \widetilde D(j)$.

    As $\widetilde D(j)$ can be considered as a subalgebra of $\gr(\widehat A) \simeq \widehat{A}$, we see that the elements $\widetilde z, \widetilde y_1, \ldots, \widetilde y_m$ are algebraically independent over $\KK$, therefore, $\mathrm{dim} \widetilde D(j) \geq m+1$. On the other hand, $\widetilde D(j)$ is a quotient of $D(j+1)$, i.e. $\mathrm{dim} \widetilde D(j) \leq \mathrm{dim} D(j+1)$. But $\mathrm{dim} D(j+1) = m + 1$ since $D(j+1)$ is an integral domain by Lemma~\ref{id2}. Then we have $\widetilde D(j) \simeq D(j+1)$.
\end{proof}

\begin{lemma}\label{l2}
    Let $D(m)$ be as in the previous lemma. Consider $D(m)$ as a subalgebra of the $\NN$-graded algebra $\widehat{A} = \oplus_{i\in \NN}\KK[\widetilde y_1, \widetilde y_{1}^{-1},\widetilde y_2, \widetilde y_2^{-1}, \ldots, \widetilde y_m, \widetilde y_m^{-1}]\widetilde z^i$.\\
    Define a $\ZZ-$filtration $\{D(m)_n\}$ on $D(m)$:
    $$D(m)_n := D(m) \cap (\oplus_{i \leq n}\widehat{A}\widetilde z^i).$$
    This filtration is admissible and the corresponding graded ring $\gr(D(m)) \simeq C$, where
    
    $$C := \frac{A}{(X_1Y_1^{k_{11}}Y_2^{k_{12}}\ldots Y_m^{k_{1m}} - Z^r, X_2Y_1^{k_{21}}Y_2^{k_{22}}\ldots Y_m^{k_{2m}} - X_1^s)}$$
\end{lemma}

\begin{remark}
     In this lemma and in the following proof we assume that $\widetilde z, \widetilde y_1, \ldots, \widetilde y_m, \widetilde x_1, \widetilde x_2$ denote the images of $Z, Y_1, \ldots, Y_m, X_1, X_2$ in $D(m)$ respectively. 
\end{remark}

\begin{proof}[Proof of Lemma~\ref{l2}]
    Note that $$D(m)_n = \{0\} \; \text{for all} \; n < 0,$$
    $$\widetilde z \in D(m)_1 \setminus D(m)_0, \widetilde x_1 \in D(m)_r \setminus D(m)_{r-1}, \widetilde x_2 \in D(m)_{rs} \setminus D(m)_{rs-1},$$
    $$ \widetilde y_1, \ldots, \widetilde y_m \in D(m)_0 \setminus D(m)_{-1}.$$
    Using the relations $\widetilde x_1 \widetilde y_1^{k_{11}}\ldots \widetilde y_m^{k_{1m}} = f_1(\widetilde z, 0,\ldots, 0), \widetilde x_2 \widetilde y_1^{k_{21}}\ldots \widetilde y_m^{k_{2m}} = \widetilde x_1^s$, we see that every element $\widetilde g \in D(m)$ can be written as follows:
    \begin{equation}\label{el2}
        \widetilde g = \sum_{i=0}^{r-1}(\sum_{0 \leq j < s}g_{ij}(\widetilde y_1, \ldots, \widetilde y_m)\widetilde x_1^j + \sum_{\substack{0 \leq j < s, \\ l > 0}}h_{ijl}(\widetilde y_1, \ldots, \widetilde y_m)\widetilde x_1^j \widetilde x_2^l) \widetilde z^i,
    \end{equation}
    where $g_{ij}(\widetilde y_1, \ldots, \widetilde y_m), h_{ijl}(\widetilde y_1, \ldots, \widetilde y_m) \in \KK[\widetilde y_1, \ldots, \widetilde y_m]$. Then the rest of the proof follows immediately from the proof of Lemma~\cite[Lemma~3.5]{GS}.
\end{proof}

Consider the following conditions for parameters $s, r, k_{ij}$:
\begin{equation}\label{star}
\begin{split}
    1) \quad & s \geq 2, k_{2j} \geq 2 \; \text{for all} \; 1 \leq j \leq m, \\
    2) \quad & r \geq 2,k_{1j} \geq 2 \; \text{for all} \; 1 \leq j \leq m.
\end{split}
\tag{*}
\end{equation}

\begin{lemma}\label{main}
    Let $C$ be the algebra, defined in the previous lemma, and one of the conditions~(\ref{star}) holds.
    Consider $C = \oplus_{i\in \ZZ} C_i$ as a graded subalgebra of the algebra $\widehat{A}$, where $$\widehat{A} = \oplus_{i\in \NN}\KK[\overline y_1, \overline y_{1}^{-1}, \ldots, \overline y_m, \overline y_m^{-1}]\overline z^i,$$ $$C_i = 0, \; \text{for} \; i < 0,$$
    $$C_i = C \cap \KK[\overline y_1, \overline y_1^{-1}, \ldots, \overline y_m, \overline y_m^{-1}]\overline z^i, \; \text{for} \; i \geq 0.$$
    Then for every homogeneous nonzero LND $\delta$, defined on $C$, we have $$\Ker(\delta) \subseteq \KK[\overline y_1, \ldots, \overline y_m].$$
\end{lemma}

\begin{remark}
 In this lemma and in the following proof we assume that $\overline z, \overline y_1, \ldots, \overline y_m, \overline x_1, \overline x_2$ denote the images of $Z, Y_1, \ldots, Y_m, X_1, X_2$ in $C$ respectively.     
\end{remark}

\begin{remark}
    We assume that $m > 1$. The case when $m = 1$ was considered in~\cite{GS}.
\end{remark}

\begin{proof}[Proof of Lemma~\ref{main}]
Let $\delta$ be a nonzero LND,which is homogeneous with respect to the grading, defined on $C$. Then elements of $C$ have the following degrees: $$\deg \overline y_1 = \ldots = \deg \overline y_m = 0, \ \deg \overline z = 1, \ \deg \overline x_1 = r, \ \deg \overline x_2 = rs.$$ Put
$$R := \frac{\KK[Y_1, \ldots, Y_m, X_1, X_2]}{(X_2Y_1^{k_{21}}\ldots Y_m^{k_{2m}} - X_1^s)}.$$
Note that $R \hookrightarrow \oplus_{i\in r\ZZ} C_i$, so $R$ can be identified with a subalgebra of $C$ identifying the images of $ Y_1, \ldots, Y_m, X_1, X_2$ with $ \overline y_1, \ldots, \overline y_m, \overline x_1, \overline x_2$ respectively. Firstly, show that $\Ker \delta \subseteq R$.

Each element $f \in C$ can be uniquely written as the following sum:
$$f = \sum_{i=0}^{r-1} f_i \overline z^i = \sum_{i=0}^{r-1} \overline z^i\sum_{j \geq 0} \widetilde f_{ij},$$
where $f_i, \widetilde f_{ij} \in R, \widetilde f_{ij}$ are homogeneous summands of $f_i$. Indeed, suppose $\overline z^i \widetilde f_{ij} = \overline z^l \widetilde f_{lp}$ for some $0 \leq i, l \leq r-1$. Then $\deg (\overline z^i \widetilde f_{ij}) = deg( \overline z^l \widetilde f_{lp})$, i.e. $\deg(\widetilde f_{ij}) + i= \deg(\widetilde f_{lp}) + l$. But $\widetilde f_{ij}, \widetilde f_{lp} \in R$, so we have:
$$\deg(\widetilde f_{ij}) - \deg(\widetilde f_{lp}) \equiv 0\ \mathrm{mod}\ r,$$
i.e. $i - l \equiv 0\ \mathrm{mod}\ r$, then $i = l$.

Suppose that $\Ker (\delta) \nsubseteq R$ and let $f \in \Ker (\delta) \setminus R$. As $\Ker \delta$ is a graded subring of $C$ and $f \notin R$, then $\widetilde f_{ij} \overline z^i \in \Ker (\delta)$ for some $\widetilde f_{ij} \in R, i > 0$. Therefore, $ \overline z^i \in R$, since a kernel of an LND is factorially closed. Using the relations $\overline x_1 \overline y_1^{k_{11}}\ldots \overline y_m^{k_{1m}} = \overline z^r$ and $\overline x_2 \overline y_1^{k_{21}}\ldots \overline y_m^{k_{2m}} = \overline x_1^s$, we have $\overline y_1, 
\ldots, \overline y_m, \overline z, \overline x_1, \overline x_2 \in \Ker (\delta)$, so $\delta = 0$, and we get a contradiction. Hence, $\Ker (\delta) \subseteq R$.

Now, let $g \in \Ker (\delta) \subseteq R$. Then $g$ can be written as the following sum:
\begin{equation}\label{hom}
    g = \sum_{0 \leq i < s}g_i(\overline y_1, \ldots, \overline y_m)\overline x_1^i + \sum_{\substack{0 \leq i < s,\\ j >0}}g_{ij}(\overline y_1, \ldots, \overline y_m)\overline x_1^i \overline x_2^j,
\end{equation}
where $g_i(\overline y_1, \ldots, \overline y_m), g_{ij}(\overline y_1, \ldots, \overline y_m) \in \KK[\overline y_1, \ldots, \overline y_m].$\\
Note that $$\deg (g_i(\overline y_1, \ldots, \overline y_m)\overline x_1^i) = ir < sr, \;\text{for} \; i<s,$$ $$\deg(g_{ij}(\overline y_1, \ldots, \overline y_m)\overline x_1^i \overline x_2^j) = ir + jrs, \;\text{for} \; 0 \leq i <s, j >0.$$

Thus every homogeneous element of $C$ in $R$ is of the form $$\text{either} \;\; g_i(\overline y_1, \ldots, \overline y_m)\overline x_1^i \;\; \text{for some} \; 0 \leq i < s,$$ $$\text{or} \;\; g_{ij}(\overline y_1, \ldots, \overline y_m)\overline x_1^i \overline x_2^j \;\; \text{for some} \; 0 \leq i <s, j >0.$$ $\Ker (\delta)$ is a graded subring of $C$, then homogeneous components of $g$ are elements of $\Ker (\delta)$, since $g \in \Ker (\delta)$.

We have $g_i(\overline y_1, \ldots, \overline y_m)x_1^i \in \Ker (\delta)$.If $i > 0$, then $\overline x_1 \in \Ker (\delta)$, since a kernel of an LND is factorially closed. However, using the relations $\overline x_1 \overline y_1^{k_{11}}\ldots \overline y_m^{k_{1m}} = \overline z^r$ and $\overline x_2\overline y_1^{k_{21}}\ldots \overline y_m^{k_{2m}} = \overline x_1^s$, we obtain the following: $\overline y_1, 
\ldots, \overline y_m, \overline z, \overline x_1, \overline x_2 \in \Ker (\delta)$. Hence, $\delta = 0$, which contradicts the assumption. Hence, in the equality~(\ref{hom}) the first sum is only consists of $g_0(\overline y_1, \ldots, \overline y_m)$.

We have  $\overline x_2 \in \Ker (\delta)$, since $g_{ij}(\overline y_1, \ldots, \overline y_m)\overline x_1^i \overline x_2^j \in \Ker (\delta)$ and since $\Ker (\delta)$ is factorially closed. Then $\delta$ extends to a nonzero LND $\widetilde \delta$ of $\widetilde R := \frac{\KK(X_2)[Z, Y_1, \ldots, Y_m, X_1]}{(X_1Y_1^{k_{11}}\ldots Y_m^{k_{1m}} - Z^r, X_2Y_1^{k_{21}}\ldots Y_m^{k_{2m}} - X_1^s)}$. Let $\widetilde z, \widetilde y_1, \ldots, \widetilde y_m, \widetilde x_1, \widetilde x_2$ denote the images of $Z, Y_1, \ldots, Y_m, X_1, X_2$ in $\widetilde R$ respectively.

Consider polynomials $F_1 := X_1Y_1^{k_{11}}\ldots Y_m^{k_{1m}} - Z^r$ and $F_2 := Y_1^{k_{21}}\ldots Y_m^{k_{2m}} - X_1^s$. We may assume that $F_1$ and $F_2$ generate an ideal $I := (X_1Y_1^{k_{11}}\ldots Y_m^{k_{1m}} - Z^r, X_2Y_1^{k_{21}}\ldots Y_m^{k_{2m}} - X_1^s)$ in $\KK(X_2)[Y_1, \ldots, Y_m, X_1]$. Define a lexicographic order on $\widetilde R$ by $X_1 < Y_1< \ldots < Y_m < Z$.

1) In the case when $s \geq 2, k_{2j} \geq 2$ we introduce $\omega$-weight filtration on $\widetilde R$: put $\deg_\omega(Z) := 1$ and put degrees of the rest of variables equal zero.

We show that $F_1, F_2$ is the Groebner basis of the ideal $I$ with respect to a lexicographic order, constructed by $\omega$. Note that $\overline{F_1}^\omega =Z^r$ and $\overline{F_2}^\omega = Y_1^{k_{21}}\ldots Y_m^{k_{2m}} - X_1^s$ are relatively prime. Then there is a sequence of elementary divisions of $S(F_1, F_2)$ by $\{F_1, F_2\}$ such that the remainder of the sequence is zero, by Lemma~\ref{s}. By Buchberger’s Criterion, $\{F_1, F_2\}$ is the Groebner basis of the ideal $I$. Using Proposition~\ref{Grebner}, we obtain: $\gr(\widetilde R) = \frac{\KK(X_2)[Z, Y_1, \ldots, Y_m, X_1]}{(Z^r, Y_1^{k_{21}}\ldots, Y_m^{k_{2m}} - X_1^s)}$. 

Let $\partial$ be a nonzero LND of $\gr(\widetilde R)$. It is easy to see that $\partial(z) = 0$. Then $\partial$ extends to a nonzero LND $\widehat\partial$ of an algebra $\widehat R := \frac{\KK(X_2)(Z)[Y_1, \ldots, Y_m, X_1]}{(Y_1^{k_{21}}\ldots Y_m^{k_{2m}} - X_1^s)}$. The algebra $\widehat R$ is rigid by~\cite[Lemma 4.5]{G}. Hence, $\widehat{\partial} = 0$ and $\partial = 0$. It follows that $\gr(\widetilde R)$ is rigid, then, by Proposition~\ref{grrig}, $\widetilde{R}$ is also rigid. That is $\widetilde\delta = 0$, so $\delta = 0$, which leads to a contradiction.

2) In the case when $r \geq 2,k_{1j} \geq 2$ we introduce $\omega$-weight filtration on $\widetilde R$: put $\deg_\omega(Z) := 1, \deg_\omega(X_1) := r$ and put degrees of the rest of variables equal zero.

We show that $F_1, F_2$ is the Groebner basis of the ideal $I$ with respect to a lexicographic order, constructed by  $\omega$. Note that $\overline{F_1}^\omega = X_1Y_1^{k_{11}}\ldots Y_m^{k_{1m}} - Z^r$ and $\overline{F_2}^\omega = X_1^s$ are relatively prime. Then there is is a sequence of elementary divisions of $S(F_1, F_2)$ by $\{F_1, F_2\}$ such that the remainder of the sequence is zero, by Lemma~\ref{s}. By Buchberger’s Criterion, $\{F_1, F_2\}$ is the Groebner basis of the ideal $I$. Using Proposition~\ref{Grebner}, we obtain: $\gr(\widetilde R) = \frac{\KK(X_2)[Z, Y_1, \ldots, Y_m, X_1]}{(X_1Y_1^{k_{11}}\ldots Y_m^{k_{1m}} - Z^r, X_1^s)}$. 

Let $\partial$ be a nonzero LND of $\gr(\widetilde R)$. It is easy to see that $\partial(x_1) = 0$. Then $\partial$ extends to a nonzero LND $\widehat\partial$ of an algebra $\widehat R := \frac{\KK(X_2)(X_1)[Z, Y_1, \ldots, Y_m]}{(X_1Y_1^{k_{11}}\ldots Y_m^{k_{1m}} - Z^r)}$. The algebra $\widehat R$ is rigid by~\cite[Lemma 4.5]{G}. Hence, $\widehat{\partial} = 0$ and $\partial = 0$. It follows that $\gr(\widetilde R)$ is rigid, then, by Proposition~\ref{grrig}, $\widetilde{R}$ is also rigid. That is $\widetilde\delta = 0$, so $\delta = 0$, which leads to a contradiction.

Since both of these cases contradict the assumption, we conclude that in the equality~(\ref{hom}) the second sum is equal to zero. Hence, $g = g_0(\overline y_1, \ldots, \overline y_m)$ and $\Ker(\delta) \subseteq \KK[\overline y_1, \ldots, \overline y_m]$.

\end{proof}

\begin{lemma}\label{lnd}
    There is a nonzero LND $\delta$ of $B$ such that $$\Ker(\delta) = \KK[y_1, \ldots, y_m].$$
\end{lemma}

\begin{proof}
    Define $\delta$ by:
    $$\delta(y_1) = \ldots = \delta (y_m) = 0,$$
    $$\delta(z) = y_1^{k_{11}+k_{21}} \ldots y_m^{k_{1m}+k_{2m}},$$
    $$\delta(x_1) = \frac{f_1(y_1^{k_{11}+k_{21}} \ldots y_m^{k_{1m}+k_{2m}}, 0, \ldots, 0)}{y_1^{k_{11}} \ldots y_m^{k_{1m}}} = P_1(y_1, \ldots, y_m),$$
    $$\delta(x_2) = \frac{f_2(y_1^{k_{11}+k_{21}} \ldots y_m^{k_{1m}+k_{2m}}, 0, \ldots, 0)}{y_1^{k_{21}} \ldots y_m^{k_{2m}}} = P_2(y_1, \ldots, y_m),$$
    where $P_1(y_1, \ldots, y_m), P_2(y_1, \ldots, y_m) \in \KK[y_1, \ldots, y_m]$. One can see that $\delta$ is an LND of $B$ and $\KK[y_1, \ldots, y_m] \subseteq \Ker(\delta)$. Moreover, $\KK[y_1, \ldots, y_m]$ is algebraically closed in $B$ and $\mathrm{tr.deg}_{\KK[y_1, \ldots, y_m]}B = 1$. Then, by Proposition~\ref{LND}, we see that $\Ker(\delta) = \KK[y_1, \ldots, y_m]$.
\end{proof}

\begin{theorem}\label{ML}
    Let $B$ be the algebra that was defined earlier, and one of the conditions~(\ref{star}) holds. Then $$\mathrm{ML}(B) = \KK[y_1, \ldots, y_m].$$
\end{theorem}

\begin{proof}
  Let $\delta$ be a nonzero LND of $B$. Consider the admissible $\ZZ$-filtration $\{B_n\}_{n \in \ZZ}$, defined in Lemma~\ref{l1}. Put $\rho_j : D(j) \rightarrow D(j+1)$, where $D(1) := B$. The map $\rho$ is a natural map from an algebra to the associated graded algebra. An LND $\widetilde \delta_1 := \delta$ induces nonzero LNDs $\widetilde \delta_2$ of $D(2), \ldots,$ $\widetilde \delta_m$ of $D(m)$ by Theorem~\ref{homgr}. Moreover, we have inclusions: $\rho_1(\Ker(\widetilde\delta_1)) \subseteq \Ker(\widetilde \delta_2), \ldots, \rho_{m-1}(\Ker(\widetilde \delta_{m-1})) \subseteq \Ker(\widetilde \delta_m)$. Consider $f \in \Ker(\delta)$ and set $\rho := \rho_1 \circ \ldots \circ \rho_{m-1}$. Replacing, if necessary, $f$ by $f - \lambda$ for an appropriate $\lambda \in \KK^{\times}$ at each step, we may assume that $\rho(f) \notin \KK$ and $\rho(f) \in \Ker(\widetilde \delta)$.

   Consider the admissible $\ZZ$-filtration $\{D(m)_n\}_{n \in \ZZ}$ that was defined in Lemma~\ref{l2}. Put $\overline \rho : D(m) \rightarrow C$. The map $\rho$ is a natural map from an algebra to the associated graded algebra. Moreover, we have an inclusion: $\overline \rho(\Ker(\widetilde \delta_m)) \subseteq \Ker(\overline \delta)$. Hence, $\overline \rho(\rho(f)) \in \Ker(\overline \delta)$. By Lemma~\ref{main}, $\Ker(\overline \delta) \subseteq \KK[\overline{y}_1, \ldots,  \overline{y}_m]$, where $\overline{y}_i$ is the image of $y_i$ in $C$. So, $\overline \rho(\rho(f)) \in \KK[\overline{y}_1, \ldots,  \overline{y}_m] \subseteq C$. The definition of the filtration of $D(m)$ from Lemma~\ref{l2} and the equation~(\ref{el2}) imply an inclusion $\rho(f) \in \KK[\widetilde{y}_1, \ldots,  \widetilde{y}_m] \subseteq D$, where $\widetilde{y}_i$ is the image of $y_i$ in $D$. Proceeding in the same manner and using the definitions of the filtrations of $D(j)$ from Lemma~\ref{l1} and the equations~(\ref{el2}), we obtain the following: $f \in \KK[y_1 \ldots, y_m] \subseteq D(1) = B$. Therefore, $\Ker(\delta) \subseteq \KK[y_1, \ldots, y_m]$. But $\Ker(\delta)$ is a factorially closed subring of $B$ and $\mathrm{tr.deg}_{\KK}\Ker(\delta) = m$. Thus $\Ker(\delta) = \KK[y_1, \ldots, y_m]$. Since $\delta$ is an arbitrary LND, we have: $\mathrm{ML}(B) = \KK[y_1, \ldots, y_m]$ by Lemma~\ref{lnd}.
\end{proof}

\begin{lemma}\label{fixit}
    For every automorphism $\phi \in \Aut(B)$ there is $\lambda \in \KK^{\times}$ such that $\phi(y_j) = \lambda_j y_i$.
\end{lemma}

\begin{proof}
    Let $C$ be a graded subalgebra of the algebra $\widehat{A}$ as in Lemma~\ref{main}. We have $\mathrm{ML}(C) = \KK[\overline y_1, \ldots, \overline y_m]$, since $C$ is a particular case of the algebra $B$. By Lemma~\ref{lnd}, there is an LND $\delta$ of $C$ such that $\Ker(\delta) = \KK[\overline y_1, \ldots, \overline y_m]$. Hence, $C$ is a semi-rigid algebra.
    
    Consider sets $U_i = \{f \in C \; | \; \deg_{\delta}(f) \leq i\}$. Then $C = \cup_i U_i$. Put $P := \KK[U_r]$. It is a subalgebra of $C$. One can see that $$\deg_{\delta} (\overline y_1)= \ldots = \deg_{\delta} (\overline y_m) = 0, \ \deg_{\delta} (\overline z) = 1, \ \deg_{\delta} (\overline x_1) = r, \ \deg_{\delta} (\overline x_2) = rs.$$ Then $P = \frac{\KK[Y_1, \ldots, Y_m, X_1, Z]}{X_1Y_1^{k_{11}}Y_2^{k_{12}}\ldots Y_m^{k_{1m}} - Z^r}$. Note that any algebra automorphism sends any LND of an algebra to an LND. The filtration $U_i$ is the same for all LNDs of $C$ as $C$ is semi-rigid. Hence, $P$ is invariant under the action of all $\xi \in \Aut(C)$. By~\cite[Lemma 6.2]{G}, we have the following relation for all $\psi \in \Aut(P)$:
    $$\psi(\overline y_j) = \mu_j \overline y_i, \; \; \mu \in \KK^{\times},  \; \; \overline y_i , \overline y_j \in P.$$
    Thus, for any $\xi \in \Aut(C)$:
    $$\xi(\overline y_j) = \eta_j \overline y_i, \; \; \eta \in \KK^{\times}.$$
    Then we have the same relation for any $\phi \in \Aut(B)$ by~\cite[Lemma 7.8]{G}:
    $$\phi(y_j) = \lambda_j y_i, \; \; \lambda \in \KK^{\times}.$$
\end{proof}

\section{Isomorphism classes}

Consider two algebras defining multi Danielewski varieties:

$$B_1 := \frac{A}{(X_1Y_1^{k_{11}}\ldots Y_{m}^{k_{1m}} - f_1(Z, Y_1, \ldots, Y_m), X_2Y_1^{k_{21}}\ldots Y_m^{k_{2m}} - f_2(Z, Y_1, \ldots, Y_m, X_1))},$$
$$B_2 := \frac{A}{(X_1Y_1^{p_{11}}\ldots Y_m^{p_{1m}} - g_1(Z, Y_1, \ldots, Y_m), X_2Y_1^{p_{21}}\ldots Y_m^{p_{2m}} - g_2(Z, Y_1, \ldots, Y_m, X_1))},$$
where $k_{ij}, p_{ij} \in \NN$, polynomials $f_1(Z, Y_1, \ldots, Y_m), g_1(Z, Y_1, \ldots, Y_m) \in \KK[Z, Y_1, \ldots, Y_m]$ are monic in $Z$; polynomials $f_2(Z, Y_1, \ldots, Y_m, X_1), g_2(Z, Y_1, \ldots, Y_m, X_1) \in \KK[Z, Y_1, \ldots, Y_m, X_1]$ are monic in $X_1$. Set $$\deg_Zf_1(Z, Y_1, \ldots, Y_m) := r_1, \deg_{X_1}f_2(Z, Y_1, \ldots, Y_m, X_1) := s_1,$$ $$\deg_Zg_1(Z, Y_1, \ldots, Y_m) := r_2, \deg_{X_1}g_2(Z, Y_1, \ldots, Y_m, X_1) := s_2.$$ 
The images of $Z, Y_1, \ldots, Y_m, X_1, X_2$ will be respectively denoted by $z, y_1, \ldots, y_m$, $ x_1, x_2$ in $B_1$ and by $\widehat z, \widehat y_1, \ldots, \widehat y_m, \widehat x_1, \widehat x_2$ in $B_2$.

\begin{theorem}\label{main1}
    Let one of the conditions~(\ref{star}) holds for every parameter $s_i, r_i, k_{ij}, p_{ij}$.\\
    Assume $B_1 \simeq B_2$. Then the following holds true:
    \begin{enumerate}
        \item[1)] $r_1 = r_2, \ s_1 = s_2$ and there is a permutation $\sigma$ of a set $\{1, \ldots, m\}$ such that $k_{j\sigma(i)} = p_{ji}$ for all $i, j$. Put $r := r_1 = r_2, s := s_1 = s_2, d_i := k_{1\sigma(i)} = p_{1i}, e_i := k_{2\sigma(i)} = p_{2i}$. 
        \item[2)] there are $\lambda_1, \ldots, \lambda_m, \gamma \in \KK^{\times}$ and polynomials $ \delta \in \KK[Y_1, \ldots, Y_m], f \in \KK[Z, Y_1, \ldots, Y_m]$ and $h \in \KK[Z, Y_1, \ldots, Y_m, X_1]$ such that 
        \begin{enumerate}
            \item[(i)] $g_1(\gamma Z + \delta, \lambda_1 Y_1, \ldots, \lambda_m Y_m) = \tau f_1(Z, Y_1, \ldots, Y_m) + Y_1^{d_1} \ldots Y_m^{d_m}f$, where $\tau = \gamma^r$.
            \item[(ii)] $g_2(\gamma Z + \delta, \lambda_1 Y_1, \ldots, \lambda_m Y_m, \nu X_1 + g) = \kappa f_2(Z, Y_1, \ldots, Y_m, X_1) + Y_1^{e_1} \ldots Y_m^{e_m}h$, where $\nu = \lambda_1^{-d_1} \ldots \lambda_m^{-d_m} \tau, \kappa = \nu^s$ and $g = \lambda_1^{-d_1} \ldots \lambda_m^{-d_m}f$.
        \end{enumerate}
    \end{enumerate}
    Moreover, if $\psi: B_2 \rightarrow B_1$ is an isomorphism, then
    $$\psi(\widehat y_i) = \lambda_i y_{\sigma(i)}, \;\;\; \psi(\widehat z) = \gamma z + \delta(y_1, 
    \ldots, y_m),$$
    $$\psi(\widehat x_1) = \nu x_1 + g(z, y_1, \ldots, y_m), \;\;\; \psi(\widehat x_2) = \lambda_1^{-e_1}\ldots \lambda_m^{-e_m}(\kappa x_2 + h(z, y_1, \ldots, y_m, x_1)).$$
    Conversely, if conditions 1) and 2) hold, then $B_1 \simeq B_2$.
\end{theorem}

\begin{proof}
Let $\psi: B_1 \rightarrow B_2$ be a $\KK$-algebra isomorphism. Replacing $B_1$ with its image, one can assume that $B_1 = B_2 = B$. By Theorem~\ref{ML}, $\mathrm{ML}(B) = \KK[y_1, \ldots, y_m] = \KK[\widehat y_1, \ldots, \widehat y_m]$. Hence, by Lemma~\ref{fixit}:
\begin{equation}\label{7}
\widehat y_i = \lambda_i y_{\sigma (i)},
\end{equation}
where $\lambda_i \in \KK^{\times}, \sigma$ is a permutation of the set $\{1, \ldots, m\}$.

Moreover, $\KK (y_1, \ldots, y_m)[z] = \KK (\widehat y_1, \ldots, \widehat y_m)[\widehat z]$. As $B \cap \KK(y_1, \ldots, y_m) = \KK[y_1, \ldots, y_m]$, we have: $\widehat z = \gamma z + \delta$ for some $\gamma, \delta \in \KK[y_1, \ldots, y_m]$. Analogous relation holds for $z$ by symmetry, thus we obtain:
\begin{equation}\label{5}
    \widehat z = \gamma z + \delta(y_1, \ldots, y_m),
\end{equation}
for $\gamma \in \KK^{\times}, \delta(y_1, \ldots, y_m) \in \KK[y_1, \ldots, y_m]$. Hence,
\begin{equation}\label{6}
    \KK[z, y_1, \ldots, y_m] = \KK[ \widehat z, \widehat y_1, \ldots, \widehat y_m].
\end{equation}
Now, show that for the permutation $\sigma$ we have $\widehat y_i = \lambda_i y_{\sigma (i)}$ if and only if $p_{ji} = k_{j\sigma(i)}$ for $j = 1, 2$. Suppose that for some $i$: $p_{1i} < k_{1\sigma(i)}$. Using the relations~(\ref{6}) and~(\ref{7}), obtain: 
$$y_{\sigma(1)}^{k_{1\sigma (1)}}\ldots y_{\sigma(m)}^{k_{1\sigma(m)}}B \cap \KK[z, y_1, \ldots, y_m] = \widehat y_1^{k_{1\sigma(1)}}\ldots \widehat y_m^{k_{1\sigma(m)}}B \cap \KK[ \widehat z, \widehat y_1, \ldots, \widehat y_m],$$
i.e.
\begin{multline*}
    (y_1^{k_{11}}\ldots y_m^{k_{1m}}, f_1(z, y_1, \ldots, y_m))\KK[z, y_1, \ldots, y_m] = \\ = (\widehat y_1^{k_{1\sigma(1)}}\ldots \widehat y_m^{k_{1\sigma(m)}}, \widehat y_1^{k_{1\sigma(1)} - p_{11}}\ldots \widehat y_m^{k_{1\sigma(m)}-p_{1m}}g_1(\widehat z, \widehat y_1, \ldots, \widehat y_m))\KK[ \widehat z, \widehat y_1, \ldots, \widehat y_m].
\end{multline*}
Therefore,
\begin{multline*}
    f_1(z, y_1, \ldots, y_m) \in \\ \in (\widehat y_1^{k_{1\sigma(1)}}\ldots \widehat y_m^{k_{1\sigma(m)}}, \widehat y_1^{k_{1\sigma(1)} - p_{11}}\ldots \widehat y_m^{k_{1\sigma(m)}-p_{1m}}g_1(\widehat z, \widehat y_1, \ldots, \widehat y_m))\KK[ \widehat z, \widehat y_1, \ldots, \widehat y_m] \subseteq \\ \subseteq y_{\sigma(i)}^{k_{1\sigma(i)}}\KK(y_{\sigma(1)}, 
    \ldots y_{\sigma(i-1)}, y_{\sigma(i+1)}, \ldots, y_{\sigma(m)} )[z, y_{\sigma(i)}],
\end{multline*}
which contradicts the fact that $f_1(Z, Y_1, \ldots, Y_m)$ is monic in $Z$. Hence, $p_{1i} \geq k_{1\sigma(i)}$ for all $i$. And, by symmetry, we have:
$$d_i := p_{1i} = k_{1\sigma(i)}.$$
Thus,
\begin{multline*}
    (y_{\sigma(1)}^{d_1}\ldots y_{\sigma(m)}^{d_m}, f_1(z, y_1, \ldots, y_m))\KK[z, y_1, \ldots, y_m] = \\ = (\widehat y_1^{d_1}\ldots \widehat y_m^{d_m}, g_1(\widehat z, \widehat y_1, \ldots, \widehat y_m))\KK[ \widehat z, \widehat y_1, \ldots, \widehat y_m].
\end{multline*}
Then $g_1(\widehat z, \widehat y_1, \ldots, \widehat y_m) = \tau'f_1(z, y_1, \ldots, y_m) + y_{\sigma(1)}^{d_1}\ldots y_{\sigma(m)}^{d_m}f'$ for some polynomials $\tau', f' \in \KK[z, y_1, \ldots, y_m]$. Since $g_1(Z, Y_1, \ldots, Y_m)$ and $f_1(Z, Y_1, \ldots, Y_m)$ are monic in $Z$, then, using the relations~(\ref{5}) and~(\ref{7}) and symmetry, we have:
$$r := r_1 = r_2$$
and $\tau' \equiv \gamma^r \; \mathrm{mod} \; y_{\sigma(1)}^{d_1}\ldots y_{\sigma(m)}^{d_m}\KK[z, y_1, \ldots, y_m]$. Let $\tau := \gamma^r$. Replacing $\tau'$ with $\tau$, we obtain:
\begin{equation*}
    g_1(\widehat z, \widehat y_1, \ldots, \widehat y_m) = \tau f_1(z, y_1, \ldots, y_m) + y_{\sigma(1)}^{d_1}\ldots y_{\sigma(m)}^{d_m}f(z, y_1, \ldots, y_m)
\end{equation*}
for some $f(z, y_1, \ldots, y_m) \in \KK[z, y_1, \ldots, y_m]$.
So for $\widehat x_1$ we have:
\begin{multline}\label{10}
    \widehat x_1 = \frac{g_1(\widehat z, \widehat y_1, \ldots, \widehat y_m)}{\widehat y_1^{d_1}\ldots \widehat y_m^{d_m}} = \frac{\tau f_1(z, y_1, \ldots, y_m) + y_{\sigma(1)}^{d_1}\ldots y_{\sigma(m)}^{d_m}f(z, y_1, \ldots, y_m)}{(\lambda_1 y_{\sigma(1)})^{d_1}\ldots (\lambda_m y_{\sigma(m)})^{d_m}} = \\ = \nu x_1 + g(z, y_1, \ldots, y_m),
\end{multline}
where $\nu = \lambda_1^{-d_1} \ldots \lambda_m^{-d_m} \tau \in \KK^{\times}$ and $g = \lambda_1^{-d_1} \ldots \lambda_m^{-d_m}f \in \KK[z, y_1, \ldots, y_m]$. Hence,
\begin{equation}\label{11}
    \KK[z, y_1, \ldots, y_m, x_1] = \KK[ \widehat z, \widehat y_1, \ldots, \widehat y_m, \widehat x_1].
\end{equation}
Now assume that for some $i$: $p_{2i} < k_{2\sigma(i)}$. Using~(\ref{7}) and~(\ref{11}), we obtain:
$$y_{\sigma(1)}^{k_{2\sigma (1)}}\ldots y_{\sigma(m)}^{k_{2\sigma(m)}}B \cap \KK[z, y_1, \ldots, y_m, x_1] = \widehat y_1^{k_{2\sigma(1)}}\ldots \widehat y_m^{k_{2\sigma(m)}}B \cap \KK[ \widehat z, \widehat y_1, \ldots, \widehat y_m, \widehat x_1],$$
i.e.
\begin{multline*}
    (y_1^{k_{21}}\ldots y_m^{k_{2m}}, f_2(z, y_1, \ldots, y_m, x_1))\KK[z, y_1, \ldots, y_m, x_1] = \\ = (\widehat y_1^{k_{2\sigma(1)}}\ldots \widehat y_m^{k_{2\sigma(m)}}, \widehat y_1^{k_{2\sigma(1)} - p_{21}}\ldots \widehat y_m^{k_{2\sigma(m)}-p_{2m}}g_2(\widehat z, \widehat y_1, \ldots, \widehat y_m, \widehat x_1))\KK[ \widehat z, \widehat y_1, \ldots, \widehat y_m, \widehat x_1].
\end{multline*}
Thus,
\begin{multline*}
    f_2(z, y_1, \ldots, y_m, x_1) \in \\ \in (\widehat y_1^{k_{2\sigma(1)}}\ldots \widehat y_m^{k_{2\sigma(m)}}, \widehat y_1^{k_{2\sigma(1)} - p_{21}}\ldots \widehat y_m^{k_{2\sigma(m)}-p_{2m}}g_2(\widehat z, \widehat y_1, \ldots, \widehat y_m, \widehat x_1))\KK[ \widehat z, \widehat y_1, \ldots, \widehat y_m, \widehat x_1] \subseteq \\ \subseteq y_{\sigma(i)}^{k_{2\sigma(i)}}\KK(y_{\sigma(1)}, 
    \ldots y_{\sigma(i-1)}, y_{\sigma(i+1)}, \ldots, y_{\sigma(m)} )[z, y_{\sigma(i)}],
\end{multline*}
which contradicts the fact that $f_2(Z, Y_1, \ldots, Y_m, X_1)$ is monic in $X_1$. So $p_{2i} \geq k_{2\sigma(i)}$ for all $i$. The converse inequality can be obtained in the similar way. Thus we have:
$$e_i := p_{2i} = k_{2\sigma(i)}.$$
Hence,
\begin{multline*}
    (y_{\sigma(1)}^{e_1}\ldots y_{\sigma(m)}^{e_m}, f_2(z, y_1, \ldots, y_m, x_1))\KK[z, y_1, \ldots, y_m, x_1] = \\ = (\widehat y_1^{d_1}\ldots \widehat y_m^{d_m}, g_2(\widehat z, \widehat y_1, \ldots, \widehat y_m, \widehat x_1))\KK[ \widehat z, \widehat y_1, \ldots, \widehat y_m, \widehat x_1].
\end{multline*}
Then $g_2(\widehat z, \widehat y_1, \ldots, \widehat y_m, \widehat x_1) = \kappa'f_1(z, y_1, \ldots, y_m, x_1) + y_{\sigma(1)}^{d_1}\ldots y_{\sigma(m)}^{d_m}h'$ for some polynomials $\kappa', h' \in \KK[z, y_1, \ldots, y_m, x_1]$.
Since $g_2(Z, Y_1, \ldots, Y_m, X_1)$ and $f_2(Z, Y_1, \ldots, Y_m, X_1)$ are monic in $X_1$, then, using the relations~(\ref{5}) and~(\ref{10}) and symmetry, we have:
$$s := s_1 = s_2$$
and $\kappa' \equiv \nu^r \; \mathrm{mod} \; y_{\sigma(1)}^{e_1}\ldots y_{\sigma(m)}^{e_m}\KK[z, y_1, \ldots, y_m, x_1]$. Let $\kappa := \nu^r$. Replacing $\kappa'$ with $\kappa$, we obtain:
\begin{equation*}
    g_2(\widehat z, \widehat y_1, \ldots, \widehat y_m, \widehat x_1) = \kappa f_2(z, y_1, \ldots, y_m, x_1) + y_{\sigma(1)}^{e_1}\ldots y_{\sigma(m)}^{e_m}h(z, y_1, \ldots, y_m, x_1)
\end{equation*}
for some $h(z, y_1, \ldots, y_m, x_1) \in \KK[z, y_1, \ldots, y_m, x_1]$.
So for $\widehat x_2$ we have:
\begin{multline*}
    \widehat x_2 = \frac{g_2(\widehat z, \widehat y_1, \ldots, \widehat y_m, \widehat x_1)}{\widehat y_1^{e_1}\ldots \widehat y_m^{e_m}} = \frac{\kappa f_2(z, y_1, \ldots, y_m, x_1) + y_{\sigma(1)}^{e_1}\ldots y_{\sigma(m)}^{e_m}h(z, y_1, \ldots, y_m, x_1)}{(\lambda_1 y_{\sigma(1)})^{d_1}\ldots (\lambda_m y_{\sigma(m)})^{d_m}} = \\ = \frac{\kappa x_2 + h(z, y_1, \ldots, y_m, x_1)}{\lambda_1^{e_1}\ldots \lambda_m^{e_m}}.
\end{multline*}
Conversely, suppose that conditions 1) and 2) hold. Consider the $\KK$-algebra map $\phi : A \rightarrow B_1$ defined by:
$$\phi(Y_i) = \lambda_i y_i, \;\;\; \phi(Z) = \gamma z + \delta(y_1, \ldots, y_m),$$
$$\phi(X_1) = \nu x_1 + g(z, y_1, \ldots, y_m), \;\;\; \phi(X_2) = \lambda_1^{-e_1}\ldots \lambda_m^{-e_m}(\kappa x_2 + h(z, y_1, \ldots, y_m, x_1)).$$
Clearly,
$$\phi(X_1Y_1^{d_1} \ldots Y_m^{d_m} - g_1(Z, Y_1, \ldots, Y_m)) = \phi(X_2Y_1^{e_1} \ldots Y_m^{e_m} - g_2(Z, Y_1, \ldots, Y_m, X_1)) = 0.$$
Thus, $\phi$ induced a surjective $\KK$-linear map $\overline \phi: B_2 \rightarrow B_1$. The map $\overline \phi$ is an isomorphism since $\dim B_1 = \dim B_2$.
\end{proof}

\section{General case}\label{Gc}

Now let $A_n$ be an algebra of the following form:
$$A_n := \KK[Z, Y_1, Y_2, \ldots, Y_m, X_1, \ldots X_n]$$
and consider a system of equations:
\begin{equation}\label{system}
\begin{cases} 
  X_1Y_1^{k_{11}}\ldots Y_m^{k_{1m}} - f_1(Z, Y_1, \ldots, Y_m) = 0,\\
  X_2Y_1^{k_{21}}\ldots Y_m^{k_{2m}} - f_2(Z, Y_1, \ldots, Y_m, X_1) = 0,\\
      \; \;\; \;\; \;\; \;\; \;\; \;\; \;\; \;\; \;\; \;\; \;\; \;\ldots\\
  X_nY_1^{k_{n1}}\ldots Y_m^{k_{nm}} - f_n(Z, Y_1, \ldots, Y_m, X_1, \ldots, X_{n-1}) = 0,
\end{cases}
\end{equation}
where $k_{ij} \in \NN$, a polynomial $f_1(Z, Y_1, \ldots, Y_m) \in \KK[Z, Y_1, \ldots, Y_m]$ is monic in $Z$; a polynomial $f_l(Z, Y_1, \ldots, Y_m,X_1, \ldots, X_{l-1}) \in \KK[Z, Y_1, \ldots, Y_m, X_1, \ldots, X_{l-1}]$ is monic in $X_{l-1}$ for $2 \leq l \leq n$. Let $I \subseteq A_n$ denote the ideal that is generated by the left-hand sides of the system~(\ref{system}). Put $B := A_n/I$.
The images of $Z, Y_1, \ldots, Y_m, X_1, \ldots, X_n$ in $B$ will be respectively denoted by $z, y_1, \ldots, y_m, x_1, \ldots, x_n$.

Set $\deg_Zf_1(Z, Y_1, \ldots, Y_m) := s_0, \; \deg_{X_{l-1}}f_l(Z, Y_1, \ldots, Y_m, X_1, \ldots X_{l-1}) := s_{l-1}$ for $2 \leq l \leq n$. We prove the generalisation of Theorem~\ref{main1} in the same manner as in previous Sections.

\begin{lemma}\label{41}
    $B$ is an integral domain.
\end{lemma}
\begin{proof}
    Prove by induction on $n$. The base case ($n = 2$) was proved in Lemma~\ref{id2}.

    Inductive step: let $J$ be an ideal generated by the left-hand sides of the first $n-1$ equations of the system~(\ref{system}); $R :=A_{n-1}/J$. $R$ is an integral domain by the base case. Identifying the images of $Z, Y_1, \ldots, Y_m, X_1, \ldots, X_{n-1} \in R$ with $z, y_1, \ldots, y_m, x_1, \ldots, x_{n-1}$ in $B$, we can assume that $R$ is a subring of $B$. Hence $$B = R[X_n]/(X_ny_1^{k_{21}}\ldots y_m^{k_{2m}} - f_n(z, y_1, \ldots, y_m, x_1, \ldots, x_{n-1})).$$
    Moreover, 
    \begin{equation*}
        R/(y_1) \simeq \left(\frac{\KK[Z, Y_2, \ldots, Y_m, X_1, \ldots, X_{n-2}]}{f_1(Z, 0, Y_2, \ldots Y_m), \ldots, f_{n-1}(Z, 0, Y_2 \ldots Y_m, X_1, \ldots, X_{n-2})}\right)[X_{n-1}].
    \end{equation*}
    Since a polynomial $f_n(Z, Y_1, \ldots, Y_m, X_1, \ldots, X_{n-1})$ is monic in $X_{n-1}$, it follows that its image in $R/(y_1)$ is not a zero-divisor on $R/(y_1)$. Hence, $y_1^{k_{n1}}$ is not a zero-divisor on $R/(f_n(z, y_1, \ldots, y_m, x_1, \ldots, x_{n-1}))$ by~\cite[Lemma~3.1]{GS}. Proceeding in the same manner, we obtain the fact that $y_2^{k_{n2}}, \ldots, y_m^{k_{nm}}$ are not zero-divisors on $R/(f_n(z, y_1, \ldots, y_m, x_1, \ldots, x_{n-1}))$. Hence, $B$ is an integral domain by Lemma~\ref{id}.
\end{proof}

\begin{lemma}\label{what}
    Put $D(1) := B$ and
    \begin{multline*}
        D(j) := A_n/(X_1Y_1^{k_{11}}\ldots Y_m^{k_{1m}} - f_1(Z, 0, \ldots, 0, Y_{j}, \ldots, Y_m), \\X_2Y_1^{k_{21}}\ldots Y_m^{k_{2m}} - X_1^{s_1}, \ldots, X_nY_1^{k_{n1}}\ldots Y_m^{k_{nm}} - X_{n-1}^{s_{n-1}}), \;\; \text{if} \;\;j > 1.
    \end{multline*}
     Consider $D(j)$, for every $j$, as a subalgebra of $\ZZ^m$-graded algebra $\hat{A} = \oplus_{i\in \ZZ^m}\KK[z]y^i$, where $y^i = y_1^{i_1}\cdots y_m^{i_m}$ and $i_k \in \ZZ$.

    Define $\ZZ-$filtration $\{D(j)_k\}$ on $D(j)$ as follows:
    $$D(j)_k := D(j) \cap (\oplus_{i \geq -k}\KK[z, y_1, y_1^{-1}, \ldots, y_{j-1}, y^{-1}_{j-1}, y_{j+1}, y^{-1}_{j+1}, \ldots y_m, y_m^{-1}]y_j^i).$$
     
    This filtration is admissible and the corresponding graded ring $\gr(D(j)) \simeq D(j+1)$.
\end{lemma}
\begin{remark}
    In this lemma and we assume that $z, y_1, \ldots, y_m$ and $ x_1, \ldots, x_n$ denote for every $D(j)$ the images of $Z, Y_1, \ldots, Y_m$ and $X_1, \ldots, X_n$ in $D(j)$ respectively.
\end{remark}
\begin{proof}[Proof of Lemma~\ref{what}]
    The proof is exactly the same as the proof of Lemma~\ref{l1}.
\end{proof}
\begin{lemma}\label{42}
    Let $D(m)$ be as in the previous lemma. Consider $D(m)$ as a subalgebra of the $\NN$-graded algebra $\widehat{A} = \oplus_{i\in \NN^m}\KK[\widetilde y_1, \widetilde y_{1}^{-1},\widetilde y_2, \widetilde y_2^{-1}, \ldots, \widetilde y_m, \widetilde y_m^{-1}]\widetilde z^i$.\\
    Define $\ZZ-$filtration $\{D(m)_k\}$ on $D(m)$:
    $$D(m)_k := D(m) \cap (\oplus_{i \leq k}\widehat{A}\widetilde z^i).$$
    This filtration is admissible and the corresponding graded ring $\gr(D(m)) \simeq C$, where
    \begin{multline*}
        C := A_n/(X_1Y_1^{k_{11}}Y_2^{k_{12}}\ldots Y_m^{k_{1m}} - Z^{s_0},\\ X_2Y_1^{k_{21}}\ldots Y_m^{k_{2m}} - X_1^{s_1}, \ldots, X_nY_1^{k_{n1}}\ldots Y_m^{k_{nm}} - X_{n-1}^{s_{n-1}}).
    \end{multline*}
\end{lemma}

\begin{remark}
    In this lemma we assume that $\widetilde z, \widetilde y_1, \ldots, \widetilde y_m$ and $ \widetilde x_1, \ldots, \widetilde x_n$ denote the images of $Z, Y_1, \ldots, Y_m, X_1, \ldots, X_n$ in $D(m)$ respectively. 
\end{remark}

\begin{proof}[Proof of Lemma~\ref{42}]
    The proof is exactly the same as the proof of Lemma~\ref{l2}.
\end{proof}

Consider the following conditions for parameters $s_l, k_{ij}$ for $m = 1$:
\begin{equation}\label{star1}
\begin{split}
    1) \quad & s_0 \geq 2, \\ or \\
    2) \quad & s_0 = 1, s_l \geq 2, k_{l1} \geq 2 \; \text{for all} \; 1 \leq l \leq n-1, \\
\end{split}
\tag{**}
\end{equation}

and for $m > 1$:
\begin{equation}\label{star2}
\begin{split}
    1) \quad & s_l \geq 2, k_{lj} \geq 2. \\
\end{split}
\tag{***}
\end{equation}

\begin{lemma}\label{main4}
    Let $C$ be the algebra, defined in the previous lemma; the condition (\ref{star1}) holds for $m =1$; for $m > 1$ the condition~(\ref{star2}) holds for all equations or for all but one equations of the system~(\ref{system}).
    
    Consider $C = \oplus_{i\in \ZZ} C_i$ as a graded subalgebra of the algebra $\widehat{A}$, where $$\widehat{A} = \oplus_{i\in \NN^m}\KK[\overline y_1, \overline y_{1}^{-1}, \ldots, \overline y_m, \overline y_m^{-1}]\overline z^i,$$ $$C_i = 0, \; \text{for} \; i < 0,$$
    $$C_i = C \cap \KK[\overline y_1, \overline y_1^{-1}, \ldots, \overline y_m, \overline y_m^{-1}]\overline z^i, \; \text{for} \; i \geq 0.$$
    Then for every homogeneous nonzero LND $\delta$, defined of $C$, we have: $$\Ker(\delta) \subseteq \KK[\overline y_1, \ldots, \overline y_m].$$
\end{lemma}

\begin{remark}
In this lemma and in the following proof we assume that $\overline z, \overline y_1, \ldots, \overline y_m, \overline x_1, \ldots, \overline x_n$ denote the images of $Z, Y_1, \ldots, Y_m, X_1, \ldots, X_n$ in $C$ respectively.     
\end{remark}

\begin{proof}[Proof of Lemma~\ref{main4}]
    Let $\delta$ be be a nonzero LND, which is homogeneous with respect to the grading, defined on $C$. Put
    $$R := \frac{\KK[Y_1, \ldots, Y_m, X_1, \ldots, X_n]}{(X_2Y_1^{k_{21}}\ldots Y_m^{k_{2m}} - X_1^{s_1}, \ldots, X_{n}Y_1^{k_{n1}}\ldots Y_m^{k_{nm}} - X_{n-1}^{s_{n-1}})}.$$
    Then we obtain the following relation in exactly the same way as in the proof of Lemma~\ref{main}: $\Ker(\delta) \subseteq R$.
    Prove by induction on $n$, that $\Ker(\delta) \subseteq \KK[\overline y_1, \ldots, \overline y_m]$.
    Base case ($n = 2$) was proved in Lemma~\ref{main} and~\cite[Lemma 3.6]{GS}.\\
    Inductive step: set $R_1 := \frac{\KK[Y_1, \ldots, Y_m, X_1, \ldots, X_{n-1}]}{(X_2Y_1^{k_{21}}\ldots Y_m^{k_{2m}} - X_1^{s_1}, \ldots, X_{n-1}Y_1^{k_{n-1,1}}\ldots Y_m^{k_{n-1,m}} - X_{n-2}^{s_{n-2}})}.$\\ Then $R = R_1[X_n]/(X_{n}Y_1^{k_{n1}}\ldots Y_m^{k_{nm}} - X_{n-1}^{s_{n-1}})$ and $\Ker(\delta_1) \subseteq \KK[\overline y_1, \ldots, \overline y_m]$, where $\delta_1$ is an extension of $\delta$ to $R_1$. Moreover, if $g_1 \in R_1$ is in the kernel of $\delta_1$, then, by the base case, $g_1 \in \KK[\overline y_1, \ldots, \overline y_m]$.

    Now, let $g \in \Ker (\delta) \subseteq R$. Then $g$ can be written as the following sum:
    \begin{equation}\label{hom1}
        g = \sum_{0 \leq i < s_{n-1}}g_i\overline x_{n-1}^i + \sum_{\substack{0 \leq i < s_{n-1},\\ j >0}}g_{ij}\overline x_{n-1}^i \overline x_n^j,
    \end{equation}
    where $g_i, g_{ij} \in R_1[\overline y_1, \ldots, \overline y_m].$\\
    Decompose $g_i$ and $g_{ij}$ into a homogeneous components in $C$. Repeating the reasoning of the proof of~\cite[Lemma 3.6]{GS} for $m=1$ and of the proof of Lemma~\ref{main} for $m > 1$, we conclude, that the relation~(\ref{hom1}) has a form $g = g_0 \in R_1$. Hence, $\delta(g) = \delta_1(g) = 0$, then, $g \in \KK[\overline y_1, \ldots, \overline y_m]$.
\end{proof}

\begin{lemma}\label{lnd4}
    There is a nonzero LND $\delta$ of $B$ such that $$\Ker(\delta) = \KK[y_1, \ldots, y_m].$$
\end{lemma}

\begin{proof}
    Define $\delta$ by:
    $$\delta(y_1) = \ldots = \delta (y_m) = 0,$$
    $$\delta(z) = y_1^{k_{11}+k_{21}} \ldots y_m^{k_{1m}+k_{2m}},$$
    $$\delta(x_1) = \frac{f_1(y_1^{k_{11}+\ldots +k_{n1}} \ldots y_m^{k_{1m}+\ldots+k_{nm}}, 0, \ldots, 0)}{y_1^{k_{11}} \ldots y_m^{k_{1m}}} = P_1(y_1, \ldots, y_m),$$
    $$\ldots$$
    $$\delta(x_n) = \frac{f_n(y_1^{k_{11}+\ldots +k_{n1}} \ldots y_m^{k_{1m}+\ldots+k_{nm}}, 0, \ldots, 0)}{y_1^{k_{n1}} \ldots y_m^{k_{nm}}} = P_n(y_1, \ldots, y_m),$$
    for $P_1(y_1, \ldots, y_m), \ldots, P_n(y_1, \ldots, y_m) \in \KK[y_1, \ldots, y_m]$. Clearly, $\delta$ is an LND of $B$ and $\KK[y_1, \ldots, y_m] \subseteq \Ker(\delta)$. As $\KK[y_1, \ldots, y_m]$ is algebraically closed in $B$ and $\mathrm{tr.deg}_{\KK[y_1, \ldots, y_m]}B = 1$, then, by Proposition ~\ref{LND}, we see that $\Ker(\delta) = \KK[y_1, \ldots, y_m]$.
\end{proof}

\begin{theorem}\label{ML1}
    Let $B$ be the algebra that was defined earlier. And the condition~(\ref{star1}) holds for $m =1$; for $m > 1$ the condition~(\ref{star2}) holds for all equations or for all but one equation of the system~(\ref{system}). Then $$\mathrm{ML}(B) = \KK[y_1, \ldots, y_m].$$
\end{theorem}

\begin{proof}
    The proof is exactly the same as the proof of Theorem~\ref{ML}, considering Lemma~\ref{41} and Lemmas~\ref{42}~\ref{main4} and~\ref{lnd4}.
\end{proof}

\begin{lemma}\label{fixit1}
    For every automorphism $\phi \in \Aut(B)$ there is $\lambda \in \KK^{\times}$ such that $\phi(y_j) = \lambda_j y_i$.
\end{lemma}

\begin{proof}
    Let $C$ be a graded subalgebra of the algebra $\widehat{A}$ as in Lemma~\ref{main4}. We have $\mathrm{ML}(C) = \KK[\overline y_1, \ldots, \overline y_m]$, since $C$ is a particular case of the algebra $B$. By Lemma~\ref{lnd4} there is an LND $\delta$ of $C$ such that $\Ker(\delta) = \KK[\overline y_1, \ldots, \overline y_m]$. Hence, $C$ is a semi-rigid algebra. 
    
    Consider sets $U_i = \{f \in C | \deg_{\delta}(f) \leq i\}$. Then $C = \cup_i U_i$. Put $P := \KK[U_{s_0}]$. It is a subalgebra of $C$. One can see that $$\deg_{\delta} (\overline y_1)= \ldots = \deg_{\delta} (\overline y_m) = 0, \ \deg_{\delta} (\overline z) = 1,$$ $$\deg_{\delta} (\overline x_1) = s_0, \ \deg_{\delta} (\overline x_2) = s_0s_1, \; \ldots, \; \deg_{\delta}(\overline x_n) = s_0 \cdot\ldots \cdot s_{n-1}.$$ Then $P = \frac{\KK[Y_1, \ldots, Y_m, X_1, Z]}{X_1Y_1^{k_{11}}Y_2^{k_{12}}\ldots Y_m^{k_{1m}} - Z^{s_0}}$. Note that any algebra automorphism sends any LND of an algebra to an LND. The filtration $U_i$ is the same for all LNDs of $C$ as $C$ is semi-rigid. Hence, $P$ is invariant under the action of all $\xi \in \Aut(C)$. By~\cite[Lemma 6.2]{G}, we have the following relation for $\psi \in \Aut(P)$:
    $$\psi(\overline y_j) = \mu_j \overline y_i, \; \; \mu \in \KK^{\times},  \; \; \overline y_i , \overline y_j \in P.$$
    Thus, for any $\xi \in \Aut(C)$:
    $$\xi(\overline y_j) = \eta_j \overline y_i, \; \; \eta \in \KK^{\times}.$$
    Then we have the same relation for any $\phi \in \Aut(B)$ by~\cite[Lemma 7.8]{G}:
    $$\phi(y_j) = \lambda_j y_i, \; \; \lambda \in \KK^{\times}.$$
\end{proof}

Consider two algebras defining multi Danielewski varieties:

\begin{equation}\label{system1}
\begin{cases} 
  X_1Y_1^{k_{11}}\ldots Y_m^{k_{1m}} - f_1(Z, Y_1, \ldots, Y_m) = 0,\\
  X_2Y_1^{k_{21}}\ldots Y_m^{k_{2m}} - f_2(Z, Y_1, \ldots, Y_m, X_1) = 0,\\
      \; \;\; \;\; \;\; \;\; \;\; \;\; \;\; \;\; \;\; \;\; \;\; \;\ldots\\
  X_nY_1^{k_{21}}\ldots Y_m^{k_{2m}} - f_n(Z, Y_1, \ldots, Y_m, X_1, \ldots, X_{n-1}) = 0
\end{cases}
\end{equation}
and
\begin{equation}\label{system2}
\begin{cases} 
  X_1Y_1^{p_{11}}\ldots Y_m^{p_{1m}} - g_1(Z, Y_1, \ldots, Y_m) = 0,\\
  X_2Y_1^{p_{21}}\ldots Y_m^{p_{2m}} - g_2(Z, Y_1, \ldots, Y_m, X_1) = 0,\\
      \; \;\; \;\; \;\; \;\; \;\; \;\; \;\; \;\; \;\; \;\; \;\; \;\ldots\\
  X_nY_1^{p_{21}}\ldots Y_m^{p_{2m}} - g_n(Z, Y_1, \ldots, Y_m, X_1, \ldots, X_{n-1}) = 0,
\end{cases}
\end{equation}
where $k_{ij}, p_{ij} \in \NN$; $f_1(Z, Y_1, \ldots, Y_m), g_1(Z, Y_1, \ldots, Y_m) \in \KK[Z, Y_1, \ldots, Y_m]$ are monic in $Z$; polynomials $f_l(Z, Y_1, \ldots, Y_m, X_{l-1})$, $ g_l(Z, Y_1, \ldots, Y_m, X_{l-1}) \in \KK[Z, Y_1, \ldots, Y_m, X_1, \ldots, X_{l-1}]$ are monic in $X_{l-1}$ for $2 \leq l \leq n$. Let $I_1 \subseteq A_n$ and $I_2 \subseteq A_n$ denote the ideals that are generated by the left-hand sides of the systems~(\ref{system1}) and~(\ref{system2}) respectively. Put $B_1 := A_n/I$ and $B_2 := A_n/I$. Set 
$$\deg_Zf_1(Z, Y_1, \ldots, Y_m) := s_{01}, \; \; \; \deg_Zg_1(Z, Y_1, \ldots, Y_m) := s_{02},$$ 
$$\deg_{X_{l-1}}f_l(Z, Y_1, \ldots, Y_m, X_1, \ldots X_{l-1}) := s_{l-1,1} \; \text{for} \; 2 \leq l \leq n,$$
$$\deg_{X_{l-1}}g_l(Z, Y_1, \ldots, Y_m, X_1, \ldots X_{l-1}) := s_{l-1,2} \; \text{for} \; 2 \leq l \leq n.$$
The images of $Z, Y_1, \ldots, Y_m, X_1, \ldots, X_n$ will be respectively denoted by $z, y_1, \ldots, y_m$ and $ x_1, \ldots, x_n$ in $B_1$ and by $\widehat z, \widehat y_1, \ldots,\widehat y_m,\widehat x_1, \ldots,\widehat x_n$ in $B_2$.

\begin{theorem}\label{main111}
    Let the condition~(\ref{star1}) hold for $m =1$ for the systems~(\ref{system1}) and~(\ref{system2}); for $m > 1$ the condition~(\ref{star2}) holds for all equations or for all but one equation of the systems~(\ref{system1}) and~(\ref{system2}).\\
    Assume $B_1 \simeq B_2$. Then the following holds true:
    \begin{enumerate}
        \item[1)] $s_{l,1} = s_{l,2}$, for $0 \leq l \leq n-1$, and there is a permutation $\sigma$ of a set $\{1, \ldots, m\}$, such that $k_{j\sigma(i)} = p_{ji}$ for all $i, j$. Put $s_l := s_{l,1} = s_{l,2}$ and $ d_{ji} := k_{j\sigma(i)} = p_{ji}$. 
        \item[2)] there are $\lambda_1, \ldots, \lambda_m, \gamma \in \KK^{\times}, \delta \in \KK[Y_1, \ldots, Y_m], h_l \in \KK[Z, Y_1, \ldots, Y_m, X_1, \ldots, X_l]$, for $1 \leq l \leq n$, such that 
        \begin{enumerate}
            \item[(i)] $g_1(\gamma Z + \delta, \lambda_1 Y_1, \ldots, \lambda_m Y_m) = \kappa_0 f_1(Z, Y_1, \ldots, Y_m) + Y_1^{d_{11}} \ldots Y_m^{d_{1m}}h_0$, where $\kappa_0 = \gamma^{s_0}, h_0 \in \KK[Z, Y_1, \ldots, Y_m]$.
            \item[(ii)] for $2 \leq l \leq n:$
            \begin{multline*}
                g_l(\gamma Z + \delta, \lambda_1 Y_1, \ldots, \lambda_m Y_m, \nu_1 X_1 + q_1, \ldots, \nu_{l-1} X_{l-1} + q_{l-1}) =\\ = \kappa_{l-1} f_l(Z, Y_1, \ldots, Y_m, X_1, \ldots, X_{l-1}) + Y_1^{d_{l1}} \ldots Y_m^{d_{lm}}h_{l-1},
            \end{multline*}
             where $1 \leq i \leq n-1$: $\nu_{i} = \lambda_1^{-d_{i1}} \ldots \lambda_m^{-d_{im}} \kappa_{i-1}, \kappa_i = \nu_i^{s_i}$ and $q_i = \lambda_1^{-d_{i1}} \ldots \lambda_m^{-d_{im}}h_{i-1}$.
        \end{enumerate}
    \end{enumerate}
    Moreover, if $\psi: B_2 \rightarrow B_1$ is an isomorphism, then
    $$\psi(\widehat y_i) = \lambda_i y_{\sigma(i)}, \;\;\; \psi(\widehat z) = \gamma z + \delta(y_1, 
    \ldots, y_m),$$
    $$\psi(\widehat x_1) = \nu_1 x_1 + q_1(z, y_1, \ldots, y_m),$$
    $$\psi(\widehat x_i) = \nu_i x_i + q_i(z, y_1, \ldots, y_m, x_1, \ldots, x_{i-1}), \; \text{for} \; 2 \leq i \leq n-1,$$
    $$\psi(\widehat x_n) = \lambda_1^{-d_{n1}}\ldots \lambda_m^{-d_{nm}}(\kappa_{n-1} x_{n} + h_{n-1}(z, y_1, \ldots, y_m, x_1, \ldots, x_{n-1})).$$
    Conversely, if conditions 1) and 2) hold, then $B_1 \simeq B_2$.
\end{theorem}
\begin{proof}
    The proof is exactly the same as the proof of Theorem~\ref{main1}, considering Theorem~\ref{ML1} and Lemma~\ref{fixit1}.
\end{proof}

\end{document}